\documentclass[12pt]{amsart}
\usepackage{verbatim,amssymb,latexsym, amscd}
\usepackage[T1]{fontenc}
\usepackage[all]{xy}
\usepackage{times}
\usepackage[
colorlinks,linkcolor=blue,citecolor=blue,urlcolor=red]{hyperref}

\renewcommand{\phi}{\varphi}
\renewcommand{\epsilon}{\varepsilon}
\newcommand{\sA}{\mathcal{A}}

\newcommand{\C}{\mathbf{C}}
\newcommand{\D}{\mathbb{D}}
\newcommand{\F}{\mathbf{F}}
\newcommand{\G}{\mathbf{G}}

\renewcommand{\L}{\mathbb{L}}
\newcommand{\N}{\mathbf{N}}
\renewcommand{\P}{\mathbf{P}}
\newcommand{\Q}{\mathbf{Q}}
\newcommand{\R}{\mathbf{R}}
\newcommand{\Z}{\mathbf{Z}}
\newcommand{\tr}{{\operatorname{tr}}}
\newcommand{\gm}{{\operatorname{gm}}}
\newcommand{\eff}{{\operatorname{eff}}}

\renewcommand{\hom}{{\operatorname{hom}}}
\newcommand{\num}{{\operatorname{num}}}
\newcommand{\proj}{{\operatorname{proj}}}
\newcommand{\et}{{\operatorname{\acute{e}t}}}
\renewcommand{\o}{{\operatorname{o}}}
\newcommand{\op}{{\operatorname{op}}}
\newcommand{\tot}{{\operatorname{tot}}}
\newcommand{\near}{{\operatorname{near}}}
\newcommand{\rd}{{\operatorname{red}}}
\newcommand{\Serre}{{\operatorname{Serre}}}
\newcommand{\Dir}{\operatorname{Dir}}
\newcommand{\Eul}{\operatorname{Eul}}
\newcommand{\DM}{\operatorname{DM}}
\newcommand{\DA}{\operatorname{DA}}
\newcommand{\Chow}{\operatorname{Chow}}
\newcommand{\Hodge}{\operatorname{Hodge}}
\newcommand{\Spec}{\operatorname{Spec}}
\newcommand{\Sm}{\operatorname{Sm}}
\newcommand{\Sch}{\operatorname{Sch}}
\newcommand{\SmCor}{\operatorname{SmCor}}
\newcommand{\Corr}{\operatorname{Corr}}
\newcommand{\Ker}{\operatorname{Ker}}
\newcommand{\Coker}{\operatorname{Coker}}
\newcommand{\IM}{\operatorname{Im}}
\newcommand{\Div}{\operatorname{Div}}
\newcommand{\End}{\operatorname{End}}
\newcommand{\Hom}{\operatorname{Hom}}
\newcommand{\uHom}{\operatorname{\underline{Hom}}}

\renewcommand{\Re}{\operatorname{Re}}
\newcommand{\Sp}{\operatorname{sp}}

\newcommand{\sF}{\mathcal{F}}
\newcommand{\sH}{\mathcal{H}}

\newcommand{\sM}{\mathcal{M}}
\newcommand{\sN}{\mathcal{N}}

\newcommand{\sT}{\mathcal{T}}

\newcommand{\ff}{\mathfrak{f}}
\newcommand{\fp}{\mathfrak{p}}

\newcommand{\un}{\mathbf{1}}

\newcommand{\by}[1]{\overset{#1}{\longrightarrow}}
\newcommand{\iso}{\by{\sim}}
\newcommand{\yb}[1]{\overset{#1}{\longleftarrow}}
\newcommand{\osi}{\yb{\sim}}
\newcommand{\inj}{\hookrightarrow}

\newcommand{\colim}{\varinjlim}

\DeclareFontFamily{U}{wncy}{}
\DeclareFontShape{U}{wncy}{m}{n}{%
<5>wncyr5%
<6>wncyr6%
<7>wncyr7%
<8>wncyr8%
<9>wncyr9%
<10>wncyr10%
<11>wncyr10%
<12>wncyr6%
<14>wncyr7%
<17>wncyr8%
<20>wncyr10%
<25>wncyr10}{}
\DeclareMathAlphabet{\cyr}{U}{wncy}{m}{n}

\newcommand{\Be}{\cyr{B}}

\swapnumbers

\newtheorem{thm}{Theorem}[section]
\newtheorem{lemma}[thm]{Lemma}
\newtheorem{prop}[thm]{Proposition}
\newtheorem{cor}[thm]{Corollary}
\theoremstyle{definition}
\newtheorem{defn}[thm]{Definition}
\newtheorem{rk}[thm]{Remark}
\newtheorem{rks}[thm]{Remarks}
\newtheorem{qn}[thm]{Question}
\newtheorem{qns}[thm]{Questions}
\newtheorem{ex}[thm]{Example}

\newcounter{spec}
\newenvironment{thlist}{\begin{list}{\rm{(\roman{spec})}}%
{\usecounter{spec}\labelwidth=20pt\itemindent=0pt\labelsep=10pt}}%
{\end{list}}

\begin{document}

\title{Zeta and $L$ functions of  Voevodsky motives}
\author{Bruno Kahn}
\address{IMJ-PRG\\Case 247\\
4 place Jussieu\\
75252 Paris Cedex 05\\
France}
\email{bruno.kahn@imj-prg.fr}
\date{Feb. 2012 and Dec. 2024}
\begin{abstract}
We associate an $L$-function $L^\near(M,s)$ to any geometric motive over a global field $K$ in the sense of  Voevodsky. This is a Dirichlet series which converges in some half-plane and has an Euler product factorisation. When $M$ is the dual of $M(X)$ for $X$ a smooth projective variety, $L^\near(M,s)$ differs from the alternating product of the zeta functions defined by Serre in 1969 only at places of bad reduction; in exchange, it is multiplicative with respect to exact triangles. If $K$ is a function field over $\F_q$, $L^\near(M,s)$ is a rational function in $q^{-s}$ and enjoys a functional equation. The techniques use the full force of Ayoub's six (and even seven) operations.
\end{abstract}
\maketitle

\hfill Preliminary version

\tableofcontents

\section*{Introduction} Let $K$ be a global field, and let $X$ be a smooth projective $K$-variety. In \cite{serre}, after a long search (see \cite[letter avril 1963, p. 143]{corr}), Serre succeeded to define local factors of zeta functions associated to the cohomology groups of $X$, generalising both Artin's $L$-functions and the Hasse-Weil zeta function associated to an elliptic curve. When $K$ is a number field, this definition has been generalised by Scholl \cite{scholl} to `mixed motives', \emph{i.e.} objects of the abelian category defined by Jannsen and Deligne with systems of realisations in \cite{jannsen} and \cite{deligne}, thus allowing for a more conceptual reformulation of Beilinson's conjectures on special values of $L$-functions  (see also Fontaine-Perrin Riou \cite{fpr} and Deninger \cite{den}).\footnote{There is no candidate for an abelian category of mixed motives in positive characteristic at present.}

These definitions use $l$-adic cohomology and, unfortunately,  depend on a still unproven conjecture: without even mentioning mixed motives, this is already the case in \cite{serre}. Namely, the local factor from \cite[\S 2]{serre} at a finite place $v$ of $K$ is of the form
\begin{equation}\label{eq0}
L_v^\Serre(H^i(X),s)=\det(1-\phi_v N(v)^{-s}\mid H^i(\bar X,\Q_l)^{I_v})^{-1}
\end{equation}
where $\phi_v$ is a geometric Frobenius at $v$, $l$ is a prime number not dividing $N(v)$ and $I_v$ is the inertia at $v$. Here $N(v)$ is the cardinality of the residue field at $v$. But for this expression to make sense as a complex function of the variable $s$, one needs at the very least that the coefficients of this inverse polynomial be complex numbers, while they are defined as $l$-adic numbers. If $v$ is a place of good reduction,  \eqref{eq0} boils down by smooth and proper base change to $\det(1-\phi_v N(v)^{-s}\mid H^i(\bar X(v),\Q_l))^{-1}$ where $X(v)$ is the special fibre of a smooth model of $X$ at $v$, which does not depend on the choice of the prime $l$ and has rational  coefficients by Deligne's proof of the ``Riemann hypothesis'' for $X(v)$ \cite{WeilI}. That this persists when $X$ has bad reduction at $v$ was proven by Terasoma in positive characteristic \cite{ter} but remains open over number fields in general \cite[$C_5$]{serre}. See \cite[5.6.3, 5.6.4]{zetaL} for a detailed discussion. 

Another issue is that the $L$-functions of \eqref{eq0} seem difficult to manipulate: if $0\to M'\to M\to M''\to 0$ is a short exact sequence of mixed motives, the equality $L^\Serre(M,s)=L^\Serre(M',s)L^\Serre(M'',s)$ fails in general.

Let  us now pass from mixed motives to \emph{triangulated motives}. Let \break  $\DM_\gm(K,\Q)$ denote Voevodsky's triangulated category of geometric motives over $K$, with rational coefficients \cite{voetri,mvw}. In this article, we associate to any object $M\in \DM_\gm(K,\Q)$ a Dirichlet series $L^\near(M,s)$, called the \emph{nearby $L$-function} of $M$, which has the following properties:
\begin{description}
\item[Convergence] $L^\near(M,s)$ has a finite abscissa of convergence.
\item[Multiplicativity] If $M'\to M\to M''\by{+1}$ is an exact triangle, then
\[L^\near(M,s)=L^\near(M',s)L^\near(M'',s).\]
In particular, $L^\near(M[1])=L^\near(M,s)^{-1}$.
\item[Tate twists] $L^\near(M(1),s) = L^\near(M,s+1)$.
\item[Euler decomposition] There is a factorisation
\[L^\near(M,s)=\prod_v L^\near_v(M,s)\]
where $v$ runs through the finite places of $K$ and $L^\near_v(M,s)$ is an \emph{$N(v)$-Euler factor}, i.e. a rational function in $N(v)^{-s}$ with $\Q$ coefficients whose zeroes and poles are $N(v)$-Weil numbers. 
\item[Inductivity] Let $L/K$ be a finite extension and $f:\Spec L\to \Spec K$ be the corresponding morphism. Then
\[L^\near(M,s)=L^\near(f_!M,s)\]
for any $M\in \DM_\gm(L,\Q)$, where $f_!:\DM_\gm(L,\Q)\to \DM_\gm(K,\Q)$ is the push-forward functor.
\item[``Normalisation''] Let $X$ be a smooth projective $K$-variety. If $v$ is a place of good reduction for $X$, then $L^\near_v(M(X)^*,s) = \zeta(X(v),s)$, where $X(v)$ is the special fibre of a smooth projective model of $X$ at $v$ (this zeta function does not depend on the choice of such a model, \cite[Prop. 5.6]{zetaL}). Here $M(X)^*$ is the dual of the motive $M(X)$ of $X$. This extends to $L^\near(\Phi(N)^*,s)$, where $N$ is a Chow motive and $v$ is a place of good reduction for $N$, where $\Phi:\Chow(K,\Q)\to \DM_\gm(K,\Q)$ is Voevodsky's functor \cite[Prop. 2.1.4]{voetri}.
\item[Rationality and functional equation] If $K$ has positive characteristic with field of constants $k\simeq \F_q$, then $L^\near(M,s)$ is a $q$-Euler factor and enjoys a functional equation
\[L^\near(M^*,1-s) = A(-q)^{-Bs} L^\near(M,s)\]
where $A\in \Q^*$ and $B\in \Z$ are some explicit numbers (Theorem \ref{t9.2}).
\item[Archimedean primes] If $K$ is a number field, to any archimedean place $v$ of $K$ is associated a ``local gamma factor'' $\Gamma_v(M,s)$ which is multiplicative in $M$ and agrees with that of a smooth projective variety $X$ as in Serre \cite[\S 3]{serre} for $M=M(X)^*$.
\end{description}

The definition of $L^\near(M,s)$ purely rests on formal properties of the rigid $\otimes$-triangulated category $\DM_\gm(K,\Q)$ and its generalisations over a base, hence, \emph{in fine}, on algebraic cycles. In particular, it is independent of $l$ \emph{a priori}. The superscript  `near' is there to point out that this definition is not very different from the classical one for smooth projective varieties, but also that it can be computed in terms of Ayoub's specialisation systems (nearby cycles): Theorem \ref{t9.1}. The quotation marks at normalisation are there because this property is sufficient, together with multiplicaitivity, only to characterise $L^\near$ up to finite products of Euler factors. It would be nice to find extra properties which make this uniqueness true on the nose. 

If $X$ is smooth projective, $L^\near_v(M(X)^*,s)$ is in general different from $L_v^\Serre(X,s)$ when $v$ is a place of bad reduction, where $L_v^\Serre(X,s)$ is the alternating product of the functions of \eqref{eq0}: see \S \ref{s9.D.2} for the example of an elliptic curve with multiplicative reduction. To explain the idea behind its definition, note that there is a competing local factor, replacing $H^i(\bar X,\Q_l)^{I_v}=H^0(I_v,H^i(\bar X,\Q_l))$ with $H^1(I_v,H^i(\bar X,\Q_l))=\break H^i(\bar X,\Q_l)_{I_v}(-1)$ (compare \cite[note 164.23]{corr}): $L^\near(M(X)^*,s)$ is a ``multiplicative average'' of these two alternating products, see Theorem \ref{t8.2} b) and Definition \ref{d9.1}. 

After getting the idea to define triangulated $L$-functions and finding how, I wondered if Grothendieck would himself have had a similar concern. And indeed he raises the issue twice, in \cite[letter 30.9.1964, top p. 196]{corr} responding to the letter quoted above where Serre formulates the question,  and in \emph{loc. cit.}, letter of 3 and 5 Oct. 1964, c) p. 202. In the first letter, he claims (unless mistaken) that the formula of \cite{serre} is indeed multiplicative\footnote{\emph{Lorsqu'on veut à tout prix une fonction $L$ qui dépende multiplicativement de $M$, il me semble hors de doute que la définition que tu préconises est la meilleure.}}. This looks strange, since taking invariants under inertia is not a (right) exact functor. In the second, he backtracks and prefers a ``définition bébête'' where $L(M,s)$ (for a mixed motive $M$) is defined as the product of $L$-functions of the factors of its semi-simplification. It will turn out, after the fact, that the latter idea is the right one for triangulated motives, see \eqref{eq9.3}. For mixed motives, however, Grothendieck's first proposal as implemented in the references given at the beginning seems to be the most interesting and the most profound.  But it depends on conjectures\dots

Thus I hope that the present construction, its unconditionality and its multiplicative property will be helpful for making progress towards the Beilinson conjectures.

This version is preliminary because I stopped at a `honest' functional equation in the sense of Grothendieck's first letter quoted above (see \cite[p. 197]{corr}). To get a nicer one as in \cite{serre} would involve giving a formula à la Grothendieck-Ogg-\v Safarevi\v c \cite[Exp. X]{SGA5}  for the Euler Poincaré characteristic of a motive of the form $f_!M$, where $f=S\to \Spec \F_q$ is a smooth projective curve. While this can obviously be done via an $l$-adic realisation, it is not clear (to me) that the local terms are independent of $l$, see Question \ref{q10.1}. The same issue arises in higher dimension when using Takeshi Saito's theory of the characteristic cycles \cite{tsaito}. Joseph Ayoub has suggested to use the Galois action on his ``full'' specialisation systems $\Psi_x$ \cite[Déf. 10.14]{ayoubetale}; I hope to come back to this in a further version.

\subsection*{Strategy} We first associate a zeta function to any motive in $\DM_\gm(k,\Q)$ when $k$ is a finite field, by using categorical traces of powers of Frobenius in this rigid $\otimes$-category. One key result is that the zeta function of the ``Borel-Moore'' motive of $X$ is the zeta function of $X$, when $X$ is an $\F_q$-scheme of finite type (Corollary \ref{c1}). We then extend this to motives over a $\Z$-scheme of finite type $S$ by the usual product formula over the closed points of $S$. Here a second key result is a trace formula, which allows us to compute these zeta functions as zeta functions of motives over $\F_q$ when $S$ is an $\F_q$-scheme (Theorem \ref{t3.1}): the proof, relying on the previous result, is almost purely motive-theoretic but we cannot avoid using the $l$-adic realisation functor (hence the trace formula of \cite{SGA5}), because of a problem of idempotents. We also get a functional equation when $S$ is proper over $\F_q$ (\emph{ibid.}). The six functors formalism established by Ayoub in his thesis \cite{ayoub} is central in these definitions and properties.

All these zeta functions are rational functions of $p^{-s}$ in characteristic $p$. In Corollary \ref{c6.1}, we show that they converge (in some half-plane) also in characteristic $0$, \emph{i.e.} when $S$ is dominant over $\Spec \Z$.

In Section \ref{s9} we arrive at the heart of the matter: the case of motives over a global field $K$. Since $\Spec K$ is not of finite type over $\Spec \Z$, the issue is to get a reasonable definition out of the previous work. We do it in two steps: first define a ``total'' $L$-function (Definition \ref{l9.1}), and then deduce the ``nearby'' $L$-function from it (Definition \ref{d9.1}), by applying Lemma \ref{l2}. In characteristic $p>0$, we get a functional equation in Theorem \ref{t9.2}. All this latter work uses heavily the six operations again, and even the seventh (the unipotent specialisation system).

\subsection*{Acknowledgements} This work has been in gestation since 1998. Since then, Science has made progress and much of this progress has been incorporated here. It has been announced at conferences a number of times, the most recent being at Regensburg in 2012 \cite{zetaLRegensburg}. The reasons of my procrastination are not entirely clear. I thank  Joseph Ayoub, Mikhail Bondarko, Fr\'ed\'eric D\'eglise, Luc Illusie and Amílcar Pacheco for helpful exchanges, with a special mention to Ayoub: he kindly wrote up his paper \cite{ayoubetale} on the $l$-adic realisations upon my request and helped me at a large number of places in this manuscript. As an exercise, the reader may count the number of times when I credit him for a proof or an idea.

\enlargethispage*{30pt}

\subsection*{Conventions and notation} As usual, a $\otimes$-category means a symmetric monoidal unital additive category, and a $\otimes$-functor is an additive strong symmetric monoidal functor. We refer to \cite[App. 8.A]{mvw} for tensor triangulated categories.

\numberwithin{equation}{section}

\section{The rigid $\otimes$-category $\DM_\gm(k,\Q)$}\label{s1}

Let $k$ be a field. We consider the category $\DM_\gm(k,\Q)$ defined by Voevodsky in
\cite{voetri} (here, with rational coefficients). It is provided  with a covariant functor ``motive''  $M:\Sm(k) \to \DM_\gm(k,\Q)$, where $\Sm(k)$ is the category of smooth separated $k$-schemes. In \cite{voetri}, two properties of
$\DM_\gm(k,\Q)$ are established when $k$ is of characteristic $0$:

\begin{itemize}
\item It is a rigid tensor pseudo-abelian triangulated category, generated by the motives of smooth projective varieties as such. We write $M^*$ for the dual of a motive $M\in \DM_\gm(k,\Q)$.
\item For any separated $k$-scheme of finite type $X$, there is an associated motive with
compact supports $M_c(X)\in \DM_\gm(k,\Q)$ such that $M_c(X)=M(X)^*(d)[2d]$ if $X$ is smooth
of pure dimension $d$, $M_c(X)=M(X)$ if $X$ is proper and, if
$Z\by{i}X$ is a closed subset with complementary open $U\by{j} X$, one has an exact triangle of the form
\begin{equation}\label{eq1.1}
M_c(Z)\by{i_*} M_c(X)\by{j^*} M_c(U)\by{+1}.
\end{equation}
\end{itemize}

\begin{thm}\label{p1.1} These properties hold for any $k$.
\end{thm}

\begin{proof} When $k$ is perfect,  see \cite[App. B]{motiftate} for the first property  and \cite[\S 5.3]{kelly} for the second. In general, let $k^p$ be the perfect closure of $k$. Then \cite[Prop. 4.5]{adjoints} or \cite{suslininsep} show that the base change functor
\[\DM_\gm(k,\Q)\to \DM_\gm(k^p,\Q)\]
is an equivalence of categories.
\end{proof}

(Theorem \ref{p1.1} is even true with coefficients $\Z[1/p]$, where $p$ is the exponential characteristic of $k$, but we won't use this refinement.)

We shall need:

\begin{lemma}\label{l1.4} Let $f:X\to Y$ be a finite, surjective morphism of smooth $k$-schemes of generic degree $d$, where $k$ is assumed to be perfect. Then $f$ induces morphisms $f^*:M_c(Y)\to M_c(X)$ and $f_*:M_c(X)\to M_c(Y)$ such that $f_*f^*=d$. Moreover,
\begin{itemize}
\item If $f$ is radicial, we also have $f^*f_*=d$. In particular $f^*$ is an isomorphism.
\item If $f$ is a Galois covering of group $G$, then $f^*$ induces an isomorphism $M_c(Y)\iso \epsilon_GM_c(X)$ where $\epsilon_G$ is the idempotent $\frac{1}{G}\sum_{g\in G} g$.
\end{itemize}
\end{lemma}

\begin{proof} By duality, we reduce to the same statement for $M(X)$ and $M(Y)$. Then they are already true on the level of finite correspondences.
\end{proof}

\begin{defn}\label{d1.1} We write $\DM_\gm^\eff(k,\Q)\subset \DM_\gm(k,\Q)$ for the full subcategory of effective geometric motives and, if $n\ge 0$, $d_{\le n}\DM_\gm^\eff(k,\Q)$ for the thick triangulated subcategory of $\DM_\gm^\eff(k,\Q)$ generated by motives of smooth varieties of dimension $\le n$.
\end{defn}

\begin{prop}\label{p1.2} If $\dim X\le n$, then $M_c(X)\in d_{\le n}\DM_\gm^\eff(k,\Q)$.
\end{prop}

\begin{proof}We may assume $k$ perfect. Induction on $n$. The case $n=0$ is clear because $M_c(X)=M_c(X_\rd)$ and $X_\rd$ is a (proper) étale $k$-scheme. Suppose $n>0$. By closed Mayer-Vietoris, we reduce to $X$ irreducible and then to $X$ a variety. By de Jong's theorem, there exists an alteration $f:\tilde X\to X$ with $\tilde X$ smooth. Choose a smooth open subset $U$ of $X$ over which $f$ is finite, and let $V=f^{-1}(U)$. By induction, $M_c(V)$ is in $d_{\le n}\DM_\gm^\eff(k,\Q)$ and so is $M_c(U)$ as a direct summand of $M_c(V)$ (Lemma \ref{l1.4}). By induction again, $M_c(X)\in d_{\le n}\DM_\gm^\eff(k,\Q)$.
\end{proof}

\begin{rk} By a similar reasoning, one can see that $d_{\le n}\DM_\gm^\eff(k,\Q)$ is generated by motives of smooth projective varieties of dimension $\le n$.
\end{rk}

\section{The case of a finite field}

\subsection{The ubiquity of Frobenius}\label{s2.A} Let $k=\F_q$ be a finite field with $q$ elements.
Consider the category $\Sch(k)$ of separated $k$-schemes of finite type, viewed as a symmetric monoidal category for the fibre product over $k$. The identity
functor of $\Sch(k)$ has a canonical $\otimes$-endomorphism: the Frobenius endomorphism, namely:
\begin{itemize}
\item Every object $X\in \Sch(k)$ has its ``absolute" Frobenius endomorphism $F_X$: $F_X$ is
the identity on the underlying space of $X$ and is given by $f\mapsto f^q$ on the structural
sheaf.
\item If $f:X\to Y$ is a morphism in $\Sch(k)$, the diagram
\[\begin{CD}
X@>F_X>>X\\
@V{f}VV @V{f}VV\\
Y@>F_Y>>Y
\end{CD}\]
commutes.
\item One has $F_{X\times_k Y} = F_X\times_k F_Y$ for any $X,Y\in \Sch$.
\end{itemize}

Starting from this situation, we can extend it to other categories associated with $k$. For
example, if $A$ is a commutative ring then $\DM_\gm^\eff(k,A)$ is obtained from the full subcategory $\Sm(k)$ of smooth $k$-schemes
via the string of functors
\begin{multline}\label{eq2.2}
\Sm(k)\to \SmCor(k,A)\to C^b(\SmCor(k,A))\\
\to K^b(\SmCor(k,A))\to K^b(\SmCor(k,A))/(MV+HI)\to DM_\gm^\eff(k,A).
\end{multline}

Here, $\SmCor(k,A)$ is the category of finite correspondences on smooth schemes with coefficients in $A$, $MV$ and $HI$
are the ``Mayer-Vietoris" and ``homotopy invariance" relations and the last step is idempotent
completion. At each step, the Frobenius endomorphism extends: for $\SmCor(k,A)$ because it
commutes with finite correspondences; on the third and fourth for formal reasons; on the fifth
because it acts on the $MV$ and $HI$ relations, and on the last once again for formal reasons.

To pass from $\DM_\gm^\eff(k,A)$ to $\DM_\gm(k,A)$ amounts to $\otimes$-inverting the Tate object
$\Z(1)$, and once again Frobenius passes through this formal operation. Similarly, we may take
coefficients in any ring $A$ rather than $\Z$.

We also get Frobenius endomorphisms on categories obtained from categories of (pre)sheaves via
the following simple trick: let $\sF$ be a contravariant functor from, say, $\Sm(k)$ to some
category. We define the Frobenius endomorphism $F_\sF$ of $\sF$ by the formula
\[F_\sF(X) = \sF(F_X)=\sF(X)\to \sF(X).\]

Suppose that $\sF$ takes values in the category of sets: in the special case where it is
representable, say $\sF=y(Y)$, one finds
\[F_{y(Y)} = y(F_Y)\]
for tautological reasons, where $y$ is the Yoneda embedding. 

\begin{rk}
This is the \emph{inverse} of the convention of \cite[XV.2.1]{SGA5}! Cf. \emph{loc. cit.} bottom p. 453.
\end{rk}

In this way, categories like $\DM^\eff(k,A)$ or $\DM^\eff_\et(k,A)$ carry their own Frobenius
automorphism, compatible with the one of $\DM_\gm^\eff(k,A)$ via the comparison functors if need
be.

\begin{ex}\label{ex2.1} We may describe $\Z(1)$ as $\G_m[-1]$. The Frobenius of the group scheme $\G_m$,
hence of the abelian sheaf $\G_m$, coincides with multiplication by $q$. Hence $F_{\Z(1)}$ is also
multiplication by $q$.
\end{ex}

Let $E/k$ be a finite extension of degree $m$, and let $f:\Spec E\to \Spec k$. We have a pair of adjoint functors
\[\DM_\gm(k,\Q)\begin{smallmatrix} f^*\\\rightleftarrows\\f_*\end{smallmatrix}\DM_\gm(E,\Q).\]

\begin{lemma}\label{l2.2a} a) For $M\in \DM_\gm(k,\Q)$, $F_{f^*M}=f^*F_M^m$.\\
b) For $M\in \DM_\gm(E,\Q)$, $F_{f_*M}^m = f_*F_M$.
\end{lemma}

\begin{proof} a) Construction \eqref{eq2.2} shows that any object $N$ of $\DM_\gm^\eff(k,\Q)$ is isomorphic to a direct summand of an object represented by a bounded complex of finite correspondences. In turn, $M$ is represented by an object of the form $(N,i)$ for $N$ as above and $i\in \Z$. The functor $f^*$ respects these constructions. Thus the statement reduces to the case $M=M(X)$ where $X$ is a smooth variety, when it is clear. Same reasoning for b).
\end{proof}


\subsection{A theorem of May} Let $\sT$ be a tensor triangulated category, i.e. a triangulated category provided with a symmetric monoidal structure which verifies axioms (TC1) -- (TC5) of May \cite[\S 4]{may}. Let $\un$ be the unit object of $\sT$. As in any symmetric monoidal category, an endomorphism $f$ of a strongly dualisable object $M$ has a \emph{trace} $\tr(f)\in \End_\sT(\un)$, defined as the composition
\begin{equation}\label{eq2.3}
\un \by{\eta} M\otimes M^*\by{f\otimes 1} M\otimes M^*\by{\sigma} M^*\otimes M\by{\epsilon}\un.
\end{equation}

Let $e$ be an endomorphism of the identity functor of $\sT$. Then
 
 \begin{thm}\label{t2.1} Let $M'\to M\to M''\by{+1}$ be an exact triangle in $\sT$. We have the equality
 \[\tr(e_M)=\tr(e_{M'})+\tr(e_{M''}).\]
 \end{thm}

\begin{proof} Although this is only stated in \cite{may} for $e$ the identity, May's proof for this special case directly generalises. More precisely, one can insert $e$'s in the diagram in the middle of \cite[p. 55]{may}, just below the occurence of the $\gamma$'s, without affecting its commutativity.
\end{proof}

\subsection{Zeta functions of endomorphisms} Let $\sA$ be a rigid additive $\otimes$-category, and let $N$ be an object of $\sA$. Let $\un$ be  the unit object of $\sA$ and let $K=\End(\un)$, assumed to be a field of characteristic $0$. In \cite[Def. 3.1]{modq}, we gave the following definition:

\begin{defn} \label{d2a}   For $f\in \End(M)$, its \emph{Z function} is
\[Z(f,t) = \exp(\sum_{n\ge 1}
\tr(f^n)\frac{t^n}{n})\in K[[t]]\]
where $\tr$ is the categorical trace described in the previous subsection.
\end{defn}

Following \cite[Def. 5.1]{modq}, we say that $\sA$ is \emph{of homological origin} if it is abelian semi-simple and if it is
$\otimes$-equivalent to $\sA'/\sN$, where $\sA'$ is a rigid $\otimes$-category admitting a strong $\otimes$-functor with values in the category of $\Z/2$-graded $L$-vector spaces for some extension $L$ of $K$,  and $\sN$ is the ideal of morphisms universally of trace $0$. By \cite[Th. 3.2, Rem. 3.3  and Th. 5.6]{modq}, we have

\begin{thm}\label{t2.3} Assume that $\sA$ is of homological origin. Then, for any $(N,f)$ as in Definition \ref{d2a}, $Z(f,t)\in K(t)$. If $f$ is invertible, we have the functional equation
\[Z(f^{-1},t^{-1})= (-t)^{\chi(M)} \det(f) Z(f,t)\]
where $\chi(M)=\tr(1_M)$ and $\det(f)$ is a certain element of $K^*$ which may be obtained by specialisation from the identity
\[\det(1-ft)= Z(f,t)^{-1}.\]
\end{thm}

In \cite[Th. A.41]{zetaL}, a slightly simpler proof of Theorem \ref{t2.3} is given, as well as a more explicit formula for $\det(f)$: if $Z(f,t)=\frac{\prod_{i=1}^m(1-\alpha_it)}{\prod_{j=1}^n(1-\beta_jt)}$ over the algebraic closure of $K$, then
\[\det(f) = \frac{\prod_{i=1}^m \alpha_i}{\prod_{j=1}^n \beta_j}.\]

\subsection{Number of points and zeta functions of pure motives} In \cite[pp. 80/81]{kleiman}, Kleiman defines the number of points and the Z function of effective homological motives, hence \emph{a fortiori} of an effective Chow motive $N=(X,p)$ where $X$ is a smooth projective $k$-variety and $p$ is a projector on $X$, by
\[N_s(N)=\langle F_X^s,{}^t p\rangle,\quad Z(N,t)=\exp(\sum_{s=1}^\infty N_s(N)\frac{t^s}{s})\]
where $\langle, \rangle$ is the intersection product.

For $f=F_N$, the Frobenius endomorphism of $N$, this expression coincides with that of Definition \ref{d2a} in view of the following lemma, which should have been in \cite{zetaL}.

\begin{lemma}\label{l2.4} a) Let $X,Y$ be two smooth projective varieties, and let $f\in \Corr(X,Y)$, $g\in \Corr(Y,X)$ be two correspondences. Then
\[\tr(g\circ f) = \langle {}^tf,g\rangle.\]
b) For $N=(X,p)$, $N_s(N)=\tr(F_N^s)$ and $Z(F_N,t)=Z(N,t)$.
\end{lemma}

\begin{proof} a) In any rigid $\otimes$-category $\sA$, we have the formula
\begin{equation}\label{eq2.4}
\tr(g\circ f) = {}^t(\iota_{AB}^{-1}f)\circ \iota_{BA}^{-1} g
\end{equation}
\cite[7.3]{akos}, where $f:A\to B$, $g:B\to A$ and $\iota_{AB}$ is the adjunction isomorphism $\Hom(\un,A^*\otimes B)\iso \Hom(A,B)$ (\emph{loc. cit.}, (6.2)) and similarly for $\iota_{BA}$. Applying this with $A=h(X)$, $B=h(Y)$ in $\Chow(k,\Q)$ and using the fact that $h(X)^*\simeq h(X)\otimes \L^{-\dim X}$, $h(Y)^*\simeq h(Y)\otimes \L^{-\dim Y}$ with $\L$ the Lefschetz motive, we get the usual formulas
\[\Hom(h(X),h(Y))\osi \Hom(\L^{\dim X},h(X\times Y)) =CH^{\dim X}(X\times Y)\otimes \Q\]
and similarly for $\Hom(h(Y),h(X))$. Thus, in the right hand side of \eqref{eq2.4}, $\iota_{BA}^{-1} g$ is simply $g$ viewed as a cycle class in $CH^{\dim Y}(Y\times X)\otimes \Q$ and ${}^t(\iota_{AB}^{-1}f)$ is $f$ viewed as a cycle class in $CH^{\dim X}(X\times Y)\otimes \Q$. By the formula for the composition of correspondences, this right hand side is $\langle {}^tf,g\rangle$.

b) We apply a) with $X=Y$, $f=p$, $g=F_X^s$, noting that $F_N^s=p\circ F_X^s$. 
\end{proof}

By Theorem \ref{t2.3} and the existence of a Weil cohomology, we get that $Z(N,t)$ is a rational function of $t$ and satisfies a functional equation of the form of this theorem. Adding the main result of \cite{WeilI}, we get a more precise result:

\begin{thm}\label{t2.4} Assume that $N$ is effective and is a direct summand of $h(X)$ where $X$ is a smooth projective variety of dimension $n$. Then the roots of the numerator and denominator of $Z(N,t)$ are effective Weil $q$-numbers of weights $\le 2n$.
\end{thm}

Here, a Weil $q$-number of weight $i$ is an element $\alpha\in \bar \Q$ such that $|\sigma(\alpha)| =q^i$ for any embedding $\sigma:\bar \Q\inj \C$; $\alpha$ is \emph{effective} if it is an algebraic integer.

\subsection{Covariance and contravariance}\label{s2.E} In \cite[Prop. 2.1.4]{voetri}, Voevodsky defined a functor
\[\Phi^\eff:\Chow^\eff(k,\Q)\to \DM_\gm^\eff(k,\Q)\]
which induces similar functors $\Phi:\Chow(k,\Q)\to \DM_\gm(k,\Q)$ and $\phi^\o:\Chow^\o(k,\Q)\to \DM_\gm^\o(k,\Q)$ where the last ones are the categories of birational motives defined in \cite{ks} and \cite{birat-tri}. These functors are covariant if one takes the \emph{covariant} convention on Chow motives: the ``graph'' functor is covariant. On the other hand, the computation of zeta functions using traces of powers of Frobenius in $\Chow(k,\Q)$ is done in the previous section with the \emph{contravariant} convention, which might conceivably be an issue.

More generally, if $\sA$ is a rigid $\otimes$-category, then the opposite\footnote{We prefer the term `opposite' to the older term `dual', which may cause confusion in this and other contexts.} category $\sA^\op$ (same objects, change the sense of morphisms) is also a rigid $\otimes$-category, but changing the sense of morphisms does not preserve the shape of \eqref{eq2.3}, which seems to create an issue. This is not the case:

\begin{lemma}\label{l2.3} Let $f:A\to A$ be an endomorphism of $A\in \sA$. Then $\tr_\sA(f)=\tr_{\sA^\op}(f)$.
\end{lemma}

\begin{proof} Consider the covariant strong $\otimes$-functor $\sA\to \sA^\op$ given by $A\mapsto A^*$: it sends $f:A\to B$ to its transpose ${}^t f:B^*\to A^*$. Thus, for $B=A$,
\[\tr_\sA(f) = \tr_{\sA^\op}({}^t f).\]

But $\tr_{\sA^\op}({}^t f)=\tr_{\sA^\op}(f)$ \cite[p. 151]{akos}.
\end{proof}

This lemma shows that the categorical trace commutes with strong $\otimes$-functors, be they covariant or contravariant.

\subsection{Number of points 
of geometric motives} As seen in \S \ref{s2.A}, the identity functor of $\DM_\gm(k,\Q)$ has a canonical $\otimes$-endomorphism: the Frobenius endomorphism, that we denote by $F$.
For any $M\in \DM_\gm(k,\Q)$, we write $F_M$ for the corresponding endomorphism of
$M$.

\begin{defn}\label{d1}For $n\in \Z$, $\sharp_n(M)=\tr(F_M^n)$; for $n=1$ we set $\sharp_1=\sharp$.
\end{defn}

\begin{thm}\label{t2.2} a) For each $n\in \Z$, $\sharp_n$ is an Euler-Poincar\'e characteristic
and defines a ring homomorphism
\[\sharp_n:K_0(\DM_\gm(k,\Q))\to\Q.\]
b) $\sharp_n$ takes values in $\Z[1/q]$; for $n\ge 1$ its restriction to $\DM_\gm^\eff(k,\Q)$ takes values in $\Z$, and induces a ring homomorphism
\[\overline\sharp_n:K_0(\DM_\gm^\o(k,\Q))\to \Z/q^n\]
where $\DM_\gm^\o(k,\Q)$ is the category of triangulated birational motives \cite{ks}.\\
c) We have the identities:
\[\sharp_n(M[1]) = -\sharp_n(M);\qquad \sharp_n(M(1)) = q^n\sharp_n(M).\]
d) If $X$ is smooth projective, then $\sharp_n(M(X))=|X(\F_{q^n})|$ for all $n\ge 1$, and $\sharp_0(M(X))$ is the Euler characteristic of $X$..
\end{thm}

Here the groups $K_0$ considered are those of triangulated categories as in \cite[VIII.2]{SGA5}.

\begin{proof} a) Multiplicativity is a general fact for rigid tensor
categories (use that $F_{M\otimes N}=F_M\otimes F_N$). Additivity follows from Theorem \ref{t2.1}. 

b) Bondarko has proven that the maps
\[K_0(\Chow^\eff(k,\Q))\by{K_0(\Phi^\eff)} K_0(\DM_\gm^\eff(k,\Q))\]
\begin{equation}\label{eq2.1}
K_0(\Chow(k,\Q))\by{K_0(\Phi)} K_0(\DM_\gm(k,\Q))
\end{equation}
induced by the corresponding functors  are bijective \cite[Th. 6.4.2 and Cor. 6.4.3]{bondarko1}. Here, the left groups are $K_0$ of \emph{additive} categories. The same holds for the homomorphism
\[K_0(\Chow^\o(k,\Q))\by{K_0(\Phi^\o)} K_0(\DM_\gm^\o(k,\Q))\]
by \cite[Prop. 8.1.1]{bondarko}. The result then follows from \cite[Th. 8.1 and 9.1]{modq}.

c) The first identity is a special case of a) (consider the exact triangle $M\to 0\to M[1]$).
The second one follows from multiplicativity and the formula $\sharp_n(\Z(1))=q^n$ (Example \ref{ex2.1}).

d) We use the strong $\otimes$-functor $\Phi:\Chow(k,\Q)\to \DM_\gm(k,\Q)$ as in b): it sends the Chow motive $h(X)$ to $M(X)$.  The conclusion follows 
by Lemma \ref{l2.3}. 
\end{proof}

\begin{lemma}\label{l2.2} Let $E,f,m$ be as before Lemma \ref{l2.2a}.   For $M\in \DM_\gm(E,\Q)$ and $n> 0$, one has
\[\tr(f_*F_M^n) =
\begin{cases}
0 &\text{if $m\nmid n$}\\
\tr(F_M^{n/m}) &\text{if $m\mid n$.}
\end{cases}
\]
For $n=0$, we have $\chi(f_*M) = [E:k]\chi(M)$.
\end{lemma}

\begin{proof}  Following a hint of J. Ayoub, we do as for induced representations:  as $f^*$ is monoidal, we have $\tr(f_*F_M^n)=\tr(f^*f_*F_M^n)$. Write $f^*f_*M=\bigoplus_{r=0}^{m-1} (\phi_E^r)^*M$, where $\phi_E$ is the Frobenius generator of $\Gamma=Gal(E/k)$. Then the matrix of $f^*f_*F_M$ relatively to this decomposition is
\[\begin{pmatrix}
0 & 0&\dots &0& (\phi_E^{m-1})^*\phi\\
\phi & 0& \dots&0& 0\\
0 & \phi_E^*\phi &\dots&0& 0\\
&&\ddots\\
0&\dots& &(\phi_E^{m-2})^*\phi&0
\end{pmatrix}\]
where $\phi:M\iso \phi_E^*M$ is a relative Frobenius morphism; moreover we have the relation
\[(\phi_E^{m-1})^*\phi\circ\dots\circ \phi_E^*\phi\circ \phi=F_M.\]

This can be checked on $M=M(Y)$, $Y$ a smooth $E$-scheme. The conclusion follows. 
\end{proof}

\subsection{Two twisting lemmas}\label{s2.F} In the next subsection, we shall need the following generalisation of a well-known lemma (cf. A.
Pacheco \cite[p. 283]{pacheco}: the idea seems to go very far back). We recall the notion of twisting by a $1$-cocycle from \cite[I.5.3 and III.1.3]{cg}.

Let $U$ be a quasi-projective $k$-variety, and let $G$ be a finite group of order $m$ acting on $U$ on the right. Then the geometric quotient $V=U/G$ exists. For each $\sigma\in G$, we define a $k$-variety $U^{(\sigma)}$ mapping to $V$ as follows:

Let $E/k$ be ``the'' extension of $k$ of order $m$, and $\Gamma=Gal(E/k)$. 
Then $\Gamma\times G$ acts on $U_E$. 
  Let $\phi$ be the arithmetic Frobenius of $k$, $\phi_E$ its image in $\Gamma$ and $H_\sigma = \langle (\phi_E^{-1},\sigma)\rangle$; we set $U^{(\sigma)}=U_E/H_\sigma$. The projection $U_E\to U\to V$ is $H_\sigma$-invariant, which defines $f^{(\sigma)}:U^{(\sigma)}\to V$. Write also $\pi_\sigma$ for the projection $U_E\to U^{(\sigma)}$. 

\begin{lemma}\label{l2.5} We have $U^{(1)}=U$, we write $f:=f^{(1)}$. For a general $\sigma$, the square
\[\begin{CD}
U_E@>\pi_\sigma>> U^{(\sigma)}\\
@VVV @VVV\\
\Spec E@>>> \Spec k
\end{CD}\]
is Cartesian; in particular, $\pi_\sigma$ is a Galois (étale) covering of group $\Gamma$.
\end{lemma}

\begin{proof} By faithfully flat descent \cite[Cor. VIII.5.4]{SGA1}, it suffices to see this after base-changing to $\Spec E$. Then the left column becomes $\Gamma\times U_E \to \Gamma\times \Spec E$. But, by construction, there exists a $k$-isomorphism $(U^{(\sigma)})_E\simeq U_E$ which converts $(\pi_\sigma)_E$ into the projection $\Gamma\times U_E\to U_E$ given by the canonical map $\Gamma\to *$. Hence the conclusion.
\end{proof}

\begin{lemma}\label{ltwist} Suppose that $f$ is a $G$-torsor (i.e. an étale Galois covering). Then we have
\[\frac{1}{m}\sum_{\sigma\in G}|U^{(\sigma)}(k)| = |V(k)|. \]
\end{lemma}

\begin{proof} 

%

We have
\[V(k)=V(\bar k)^\phi.\]

The map $f:U(\bar
k)\to V(\bar k)$ is surjective and $G$ acts simply transitively on its fibres. 
Let $x\in V(k)$. Pick $y\in U(\bar k)$ mapping to $x$. Then $\phi^{-1}y$ is in the fibre of
$x$, which implies that $y=\phi y\sigma $ for a unique $\sigma\in G$. Thus $y$
defines a $k$-rational point of $U^{(\sigma)}$. If $y'\in f^{-1}(x)$ is another point, then
$y'=y\tau $ for a unique $\tau\in G$, and 
\[y'=\phi y\sigma\tau =\phi y'\tau^{-1}\sigma\tau \]
so $y'\in U^{(\tau^{-1}\sigma\tau)}(k)$, and $y'\in U^{(\sigma)}(k)$ if and only if $\tau\in
Z_{G}(\sigma)$.

Summarising: there exists a well-defined conjugacy class $\bar \sigma(x)\subset G$ such that
$f^{-1}(x)$ is a disjoint union of subsets $f^{-1}(x)_\sigma$ consisting of $r$ elements of
$U^{(\sigma)}(k)$ for $\sigma$ running through $\bar \sigma(x)$, where $r=|Z_{G}(\sigma)|=|G|/|\bar \sigma(x)|$. If $\bar \sigma$ is now a given conjugacy class, let 
\[V(k)_{\bar\sigma}=\{x\in V(k)\mid \bar \sigma(x)=\bar \sigma\} \; \text{(possibly empty!)}.\] 

Then, for every $\sigma\in \bar \sigma$, the map $U^{(\sigma)}(k)\to V(k)_{\bar \sigma}$ is a torsor under
$Z_{G}(\sigma)$. In particular, $|U^{(\sigma)}(k)| = r|V(k)_{\bar \sigma}|$ and 
\[\sum_{\sigma\in \bar\sigma} |U^{(\sigma)}(k)| = |\bar \sigma|r |V(k)_{\bar \sigma}| =
m|V(k)_{\bar \sigma}|. \] 

Collecting over the conjugacy classes of $G$, we get (iii).
\end{proof}

We now need a pendant of Lemma \ref{ltwist} for traces of Frobenius. 

\begin{lemma}\label{ltwisttr} With notation and hypotheses as in Lemma \ref{ltwist}, we have
\[\frac{1}{m}\sum_{\sigma\in G}\sharp(M_c(U^{(\sigma)})) = \sharp(M_c(V)). \]
\end{lemma}

\begin{proof} Keep the notation in the proof of Lemma \ref{ltwist}. For $\sigma\in G$, let $\epsilon_\sigma= \frac{1}{m}\sum_{h\in H_\sigma} h\in \Q[\Gamma\times G]$: then, by Lemma \ref{l1.4},
\[M_c(U^{(\sigma)}) \simeq \epsilon_\sigma M_c(U_E).\]

We have $M_c(U_E)= M_c(\Spec E) \otimes M_c(U)=M(\Spec E) \otimes M_c(U)$, hence $F_{M_c(U_E)}=F_{M(\Spec E)}\otimes F_{M_c(U)}$. The endomorphism $F_{M(\Spec E)}$ coincides with the Frobenius automorphism $\phi_E\in Gal(E/k)$ acting on $M(\Spec E)$. Hence
\begin{multline*}
\sum_{\sigma\in G}\tr(F_{M_c(U^{(\sigma)})}) =\tr\left(\sum_{\sigma\in G}\phi_E\otimes  \epsilon_\sigma F_{M_c(U)}\right)\\= \tr\left(\sum_{\sigma\in G} \sum_{r=0}^{m-1} \frac{1}{m}\phi_E\otimes (\phi_E^{-r},\sigma^r) F_{M_c(U)}\right)\\ = \frac{1}{m}\tr\left(\sum_{\sigma\in G} \sum_{r=0}^{m-1} \phi_E^{1-r}\otimes\sigma^r F_{M_c(U)}\right)= \frac{1}{m}\left(\sum_{\sigma\in G} \sum_{r=0}^{m-1})\tr(\phi_E^{1-r})\tr(\sigma^r F_{M_c(U)})\right).
\end{multline*}

But $\tr(\phi_E^{1-r})=0$ for $1-r\ne m$ and $\tr(\phi_E^{m})=m$ (see Lemma \ref{l2.2}). Hence the last sum collapses to
\[\frac{1}{m}\left(\sum_{\sigma\in G} m\tr(\sigma^{1-m} F_{M_c(U)})\right)=m\tr\left(\frac{1}{m}\sum_{\sigma\in G} \sigma F_{M_c(U)}\right)=m\tr(F_{M_c(V)})\]
as requested.
\end{proof}

\begin{qn} Does the equality $\sum_{\sigma\in G}[U^{(\sigma)}] = m[V]$ hold in the Grothendieck group of varieties? It would yield lemmas \ref{ltwist} and \ref{ltwisttr} in one gulp.
\end{qn}

\subsection{More general schemes} Theorem \ref{t2.2} d) extends to

\begin{thm}\label{p1}
If $X$ is a separated $k$-scheme of finite type, then
\[\sharp_n(M_c(X))=|X(\F_{q^n})| \text{ for all } n\ge 1.\]
\end{thm}

We first give a lemma:

\begin{lemma} \label{l1.2} Consider an open-closed situation
\begin{equation}
Z\by{i} X\yb{j} U
\end{equation}
where $Z$ is a closed subset of $X$ with open complement $U$. Then, if Theorem \ref{p1} is true for two among $X,Z,U$, it is true for the third.
\end{lemma} 

\begin{proof} This follows from  Theorem \ref{p1.1} and Theorem \ref{t2.2} a).
\end{proof}

Before starting the proof, let us explain why it is going to be complicated. Suppose that we know resolution of singularities over $k$. Then we can easily use Lemma \ref{l1.2} to reduce to $X$ smooth projective. In its absence, we want of course to use de Jong's theorem on alterations. But if we try to do a proof as for Proposition \ref{p1.2}, we run into the following problem: if $U\to V$ is an étale covering of smooth varieties and Theorem \ref{p1} is true for $U$, why should it be true for $V$? The fact that $M_c(V)$ is isomorphic to a direct summand of $M_c(U)$ does not help here because, unlike abelian groups, rational integers do not have direct summands\dots\ (In summary: for numbers, the devil is in the idempotents.) Fortunately, the twisting lemmas proven in the previous subsection will help us get around this issue.

\begin{proof}[Proof of Theorem \ref{p1}]

We argue by induction on $\dim X$. The $0$-dim\-en\-sio\-nal case follows from Theorem \ref{t2.2} d) (or is trivial). Suppose $\dim X>0$.  By using Lemma \ref{l1.2}, we first reduce to $X$ a variety and then (by Nagata's theorem) to $X$ proper. 
By de Jong's equivariant alteration theorem \cite[Th. 7.3]{dJ}, there exists then a quasi-Galois alteration $f:\tilde X \to X$, with $\tilde X $ smooth projective. 
Let $V\subseteq X$ be a smooth open subset over which $f$ is finite, and let $U=f^{-1}(V_0)$. 

For simplicity, write $k_n$ for $\F_{q^n}$. We have $\sharp_n(M_c(U))=|U(k_n)|$ by Theorem \ref{t2.2} d), Lemma \ref{l1.2} and induction. 

 Let $G$ be the Galois group of $f$ and $X_G$ be the geometric quotient of $\tilde X$ by $G$. 
%
%
Write $f':\tilde X\to X_G$ for the corresponding factorisation of $f$. If we twist with respect to $f'$ in the style of \S \ref{s2.F}, then  $\tilde X^{(\sigma)}$ is smooth projective for all $\sigma\in G$ as a consequence of Lemma \ref{l2.5}, and we also have $\sharp_n(M_c(U^{(\sigma)}))=|U^{(\sigma)}(k_n)|$. 

Putting Lemmas \ref{ltwist} and \ref{ltwisttr} together then yields $\sharp(M_c(U_G))\allowbreak=|U_G(k)|$. Since $U_G\to V$ is finite and radicial, Lemma \ref{l1.4} shows that the same holds for $V$. Finally we get $\sharp(M_c(X))=|X(k)|$ by induction and Lemma \ref{l1.2} again.
\end{proof}

\begin{rk} It would be more reasonable to give a direct proof of Theorem \ref{p1}, in the style of Lemma \ref{l2.4}. Unfortunately I haven't been able to find such a proof.
\end{rk}

\subsection{The zeta function} Recall the important formula \cite[I, (3.2.3.6)]{saa}
\[F_{M^*}={}^t F_M^{-1}.\] 


In view of the above computations, it would be most natural to define the zeta function of $M$ as $Z(F_M,t)$, so that $Z(X,t)=Z(M_c(X),t)$ by Proposition \ref{p1} b). However this would create awkward functoriality problems in the next section, when we deal with motives over a base: see Theorem \ref{t3.2} below. 
One may think of these problems as arising because $S\mapsto \DM_\gm(S,\Q)$ is a \emph{homology} theory, whereas the functorial behaviour of zeta and $L$-functions expresses itself naturally in \emph{cohomological} terms. Our solution to this issue is to give a slightly artificial definition of the zeta function:

\begin{defn}\label{d2} $Z(M,t)=Z(F_M^{-1},t)=Z(F_{M^*},t)$.
\end{defn}

\begin{thm}\label{t2} $M\mapsto Z(M,t)$ is multiplicative on exact triangles, hence defines a homomorphism $K_0(\DM_\gm(k,\Q))\to 1+t\Q[[t]]$. For any $M\in \DM_\gm(k,\Q)$,\\
a) $Z(M,t)\in \Q(t)$. The degree of this rational function is $-\chi(M)$, where $\chi(M)=\tr(1_M)$.\\
b) We have the functional equation
\[Z(M^*,t^{-1}) = (-t)^{\chi(M)}\det(F_M)^{-1}Z(M,t)\] 
 where $\det(F_M)$ is the value at $t=\infty$ of $(-t)^{\chi(M)} Z(M,t)^{-1}$. 
 \\
 c) We have the identities
 \[Z(M[1],t) = Z(M,t)^{-1},\qquad Z(M(1),t) = Z(M,q^{-1}t).\]
 d) For any $f:X\to \Spec k$ of finite type, we have
 \[Z(M^{BM}(X),t) = Z(X,t).\]
\end{thm}

\begin{proof} The first fact follows from Theorem \ref{t2.2}. Using now the surjectivity of \eqref{eq2.1}, we reduce to the case where $M\in
\Chow(k,\Q)$. Then we can compute in $\sM_\num(k,\Q)$ and everything follows from Theorem \ref{t2.3} (see remarks before Theorem \ref{t2.4}). d) follows from Theorem \ref{p1}.
\end{proof}

\begin{rk} It is likely that Theorem \ref{t2} extends to all zeta functions of endomorphisms as
in Definition \ref{d2} a), but this seems to require extending Bondarko's theorem to $K_0$ of
the categories of endomorphisms as in \cite[\S 15]{modq}. (He has done that!)
\end{rk}

\begin{prop} Let $f:\Spec E\to \Spec k$ be a finite extension, and let
$f_*:\DM_\gm(E)\to \DM_\gm(k)$ be the corresponding direct image functor. Then, for
any $M\in \DM_\gm(E)$, we have
\[Z(f_* M,t) = Z(M,t^m)\]
where $m=[E:k]$.
\end{prop}

\begin{proof} This is clear from Lemma \ref{l2.2a} b) and the definition of $Z$.
\end{proof}

\begin{defn}\label{d2.1} We denote by  $\Dir$ the group of formal Dirichlet series beginning with $1$.
\end{defn}

\begin{cor}\label{c1} For $M\in \DM_\gm(k,\Q)$, let
\[\zeta(M,s) = Z(M, q^{-s}).\]
Then $M\mapsto \zeta(M,s)$ is multiplicative on exact triangles, hence defines a homomorphism $K_0(\DM_\gm(k,\Q))\to \Dir$. Moreover,\\
a) $\zeta(M,s)$ is a rational function in $q^{-s}$, of degree $-\chi(M)$.\\
b) We have the functional equation
\[\zeta(M^*,-s) = (-q^{-s})^{\chi(M)}\det(F_M)\zeta(M,s).\] 
c) We have the identities
 \[\zeta(M[1],s) = \zeta(M,s)^{-1},\qquad \zeta(M(1),s) = \zeta(M,s+1).\]
 d) For any $f:X\to \Spec k$ of finite type, we have
 \[\zeta(M^{BM}(X),s) = \zeta(X,s).\]
e) For a finite extension $f:\Spec E\to \Spec k$, we have $\zeta(M,s)=\zeta(f_* M,s)$.\qed
\end{cor}

We shall need the following

\begin{prop}\label{p3} Let $M^*\in d_{\le n}\DM_\gm^\eff(k,\Q)$ (see Definition \ref{d1.1}). Then the zeroes and poles of
$Z(M,t)$ are  effective Weil $q$-numbers of weights $\in [0,2n]$. In particular, the Dirichlet series $\zeta(M,s)$ converges absolutely for $\Re(s)>n$.
\end{prop}

\begin{proof} Consider the diagram
\[\begin{CD}
K_0(\Chow^\eff(k,\Q))@>\phi>\sim> K_0(\DM_\gm^\eff(k,\Q))\\
@V{\psi}VV\\
K_0(\sM_\num^\eff(k,\Q)).
\end{CD}\]

The dimension filtration induces a filtration on the three $K_0$'s, and $\phi$ and $\psi$
respect these filtrations. So does the inverse of $\phi$, because it is induced by Bondarko's
functor
\[\DM_\gm^\eff(k,\Q)\to K^b(\Chow^\eff(k,\Q))\]
which obviously respects the dimension filtrations. 

Since $M\mapsto Z(M,t)$ factors through $K_0(\sM_\num^\eff(k,\Q))$, we are therefore reduced to Theorem \ref{t2.4}.
\end{proof}


\begin{cor} \label{c2.1} if $\dim X\le n$,  the zeroes and poles of $Z(X,t)$ are  effective Weil $q$-numbers of weights $\in [0,2n]$.\qed
\end{cor}

Of course, this corollary may also be deduced from \cite{WeilII}.

\section{Motives over a scheme of finite type} 

\subsection{Motives over a base} In the sequel, we shall need a theory of triangulated motives over general base schemes, with a formalism of six (even seven) operations as in \cite{ayoub}. Unfortunately, it is unknown whether the natural generalisation of Voevodsky's construction enjoys such a formalism \cite[6.2]{ffihes}. This can be corrected by using the subcategory of constructible objects in Ayoub's étale motives without transfers $\DA^\et(-,\Q)$ \cite{ayoubetale} or Cisinski-Déglise Beilinson motives $\DM_\Be(-)$ \cite{cis-deg}, which coincide anyway  \cite[6.3 and 6.4]{ffihes}: the important point here is that these subcategories are preserved under the six operations (\cite[Scholie 2.2.34 and Th. 2.2.37]{ayoub} for $f^*,f_*,f_!,f^!$, \cite[Prop. 2.3.62]{ayoub} for $\uHom$). We adopt here the viewpoint of \emph{loc. cit.}, 6.5, using only the \emph{existence} of a formalism of six operations for some categories $\D(S)$ which agree with $\DM_\gm(k,\Q)$ when $S=\Spec k$, and are provided with an $l$-adic realisation functor which commutes with the six operations. Very occasionally, we shall use non-formal properties of such a theory.

We write $\Z_S=\Z$ for the unit object of $\D(S)$. If $f:X\to S$ is a smooth $S$-scheme (separated and of finite type), we write $M_S(X)=f_\# \Z_X\in\D(S)$ for its \emph{motive}, where $f_\#$ is the left adjoint of $f^*:\D(S)\to \D(X)$ \cite[1.4.1, axiom 3]{ayoub}. 

\subsection{The Borel-Moore motive revisited} This subsection partly answers  \cite[Rem. 6.7.3 3)]{ffihes}.

\begin{thm}\label{t3.2} Let $k=\F_q$ and let $f:X\to \Spec k$ be a separated $k$-scheme of finite type. Then we have an isomorphism
\[M^{BM}(X)\simeq f_!\Z_X\]
at least if $X$ is embeddable in a smooth scheme (e.g. if $X$ is quasi-projective). If $f$ is smooth, this isomorphism is canonical and natural for open immersions.
\end{thm}

\begin{proof}  Suppose first that $f$ is smooth of (pure) dimension $d$. Applying \cite[Vol. I, Scholie 1.4.2 3]{ayoub}, we find
\begin{multline}\label{eq3.4}
f_!\Z_X\simeq f_\# Th^{-1}(\Omega_f) \Z_X\\
=f_\# \Z_X(-d)[-2d]=M(X)(-d)[-2d]\simeq M^{BM}(X).
\end{multline}

Here $Th(\Omega_f)$ is the Thom equivalence associated to the module of differentials $\Omega_f$, which is computed to be the said Tate twist in \cite[2.4.38 and 2.4.40]{cis-deg}, plus the description of $M_c(X)$ for $X$ smooth given at the beginning of Section \ref{s1}. This isomorphism clearly commutes with open immersions. 

Suppose now that $X$ can be embedded as a closed subscheme of a smooth scheme $Y$. For simplicity, write $\tilde M^{BM}(X)=f_!\Z_X$. We then get an isomorphism $\tilde M^{BM}(X)\iso M^{BM}(X)$ by completing the isomorphisms in the diagram of exact triangles
\[\begin{CD}
\tilde M^{BM}(U)@>>> \tilde M^{BM}(Y)@>>> \tilde M^{BM}(X)@>+1>>\\
@V\wr VV @V\wr VV\\
M^{BM}(U)@>>> M^{BM}(Y)@>>> M^{BM}(X)@>+1>>
\end{CD}\]
where $U=Y-X$, the top row is the dual of \eqref{eq1.1} and the bottom row is obtained by applying $g_!$ to the localisation exact triangle \cite[Vol. 1, p. 77]{ayoub}
\begin{equation}\label{eq3.1}
j_!j^!\Z_Y\to \Z_Y\to i_*i^*\Z_Y\by{+1}
\end{equation}
where $g:Y\to \Spec k$ is the structural morphism and $j:U\to Y$, $i:X\to Y$ are the open and closed immersion. (Note that $j^!=j^*$ and $i_*=i_!$.)
\end{proof}

\begin{rk}\label{r1.1} It would be more reasonable and more efficient to define \emph{a priori} a natural morphism
\[M_c(X)= C_*(\Z_\tr^c(X))\to (f_!\Z_X)^*\]
in $\DM_\gm(k)$, where $\Z_\tr^c(X)(U)=z(U,X)$ is the group of quasi-finite correspondences, and to show that it is an isomorphism by reduction to the smooth (or smooth projective) case. By duality \cite[Th. 2.3.75]{ayoub}, the right hand side can be written 
\[(f_!\Z_X)^*=D_k(f_!\Z_X)\simeq f_*D_X(\Z_X)=f_*f^!\Z.\]
This amounts to defining a map
\[f^*C_*(\Z_\tr^c(X))\to f^!\Z\]
and I don't know how to construct it\dots
\end{rk}

\subsection{Zeta functions}

Let $f:S\to \Spec \Z$ be a scheme of finite type over $\Z$.  For each $x\in S_{(0)}$, we have a  pull-back functor
\[i_x^*:\D(S)\to \D(\kappa(x))\simeq \DM_\gm(\kappa(x),\Q).\]

\begin{prop}\label{p3.1} For $M\in \D(S)$, the product
\[\zeta(M,s) = \prod_{x\in S_{(0)}} \zeta(i_x^* M,s)\] 
is convergent in the group $\Dir$ (Definition \ref{d2.1}) for the topology given by the order of the first nonzero term.  It is multiplicative on exact triangles and we have the identities
 \[\zeta(M[1],s) = \zeta(M,s)^{-1},\qquad \zeta(M(1),s) = \zeta(M,s+1).\]
\end{prop}

\begin{proof} Since the exponential is continuous on $\Dir$ for the said topology, it suffices to show that the sum
\[\sum_{x\in S_{(0)}} \sum_{n\ge 1} \sharp_n(i_x^*M) \frac{N(x)^{-ns}}{n}\]
is convergent. Here, $N(x)$ is the cardinal of the residue field $\kappa(x)$. Rearranging, we must show that, for any prime $p$ and any $r\ge 1$, the set
\[\{x\in S_{(0)}\mid n[\kappa(x):\F_p]=r\}\]
is finite. This is clear, since $S$ has only a finite number of closed points of a given degree over $\F_p$. The identities follow from the case of finite fields (Corollary \ref{c1}).
\end{proof}

\subsection{The case of characteristic $p$} In this subsection, we assume that $S$ is an $\F_q$-scheme. For simplicity we still write $f$ for the structural morphism $X\to \Spec \F_q$.

\begin{prop}\label{p3.3} For $n\ge 0$, let $d^!_{\le n}\D(S)$ be the thick triangulated subcategory of $\D(S)$ generated by the $g_!\Z_X$, where $g:X\to S$ runs through the morphisms of relative dimension $\le n$. Then $(f_!M)^*\in d_{\le n+d}\DM^\eff(S,\Q)$ for any $M\in d^!_{\le n}\D(S)$, where $d=\dim S$. Moreover, for any $M\in \D(S)$ there exists $n\ge 0$ and $r\in \Z$ such that $M(r)\in d^!_{\le n}\D(S)$.
\end{prop}

\begin{proof} It suffices to show that $(f_!g_!\Z_X)^*\in d_{\le n+d}\DM^\eff(S,\Q)$ when $\dim g\le n$. By Theorem \ref{t3.2}, this is $M_c(X)$; the claim then follows from Proposition \ref{p1.2}. It suffices to prove the last statement for generators $M_S(X)(s)$, with $g:X\to S$ smooth. But $g_!\Z_X\simeq M_S(X)(-n)[-2n]$ if $n=\dim g$ as in \eqref{eq3.4}, so $M_S(X)(s)(-n-s)\in d^!_{\le n}\D(S)$.
\end{proof}


For any $x\in S_{(0)}$, we have a specialisation functor
\[\Sp_x=(f_x)_*i_x^*:\D(S)\to \DM_\gm(\F_q,\Q)\]
where $f_x:\Spec \kappa(x)\to \Spec \F_q$ is the structural morphism. We have
\begin{equation}\label{eq3.3}
\zeta(i_x^*M,s)= \zeta(\Sp_x M,s)
\end{equation}
by Corollary \ref{c1}.


From Lemma \ref{l2.2}, we get:

\begin{lemma}\label{l3.1} For any $M\in \D(S)$ and any $n\ge 1$, the function
\[S_{(0)}\ni x\mapsto \sharp_n(\Sp_x M) 
\]
is $0$ outside the finite set $\bigcup_{m\mid n} S(\F_{q^m})$.
\end{lemma}

This gives a meaning to the following proposition. For simplicity, we write
\begin{equation}\label{eq3.2}
\sharp_n^*(M)=\sharp_n(M^*) = \sharp_{-n}(M)
\end{equation}
for $M\in \DM_\gm(\F_q)$.

\begin{prop}[trace formula]\label{p3.2} We have
\[\sharp^*_n(f_!M) =\sum_{x\in S_{(0)}} \sharp^*_n(\Sp_x M)\] 
for any $M\in \D(S)$ and any $n\ge 1$.
\end{prop}

\begin{proof}We give two proofs, one in a special case and one in general: 

1) The case where $M=g_!\Z_X$ for $g:X\to S$ a morphism of finite type, where $X$ is embeddable in a smooth $k$-scheme. Then
\[\sharp_n^*(f_!M) = \sharp_n^*((fg)_!\Z_X) = \sharp_n^*(M^{BM}(X)) =\sharp_n(M_c(X)) = |X(\F_{p^n})|\]
by Theorems \ref{t3.2} and \ref{p1}. On the other hand, let $X_x=g^{-1}(x)$. The base change theorem \cite[scholie 1.4.2 1)]{ayoub} applied to the Cartesian square
\begin{equation}\label{eq3.5}
\begin{CD}
X_x@>I_x>> X\\
@Vg_x VV @Vg VV\\
x@>i_x>> S
\end{CD}
\end{equation}
yields\footnote{This is part of the definition of a crossed functor, for which the reader should consult \cite[Déf. 1.21.12]{ayoub}.}
\[i_x^*g_!\Z_X\simeq (g_x)_! I_x^*\Z_X =(g_x)_! \Z_{X_x}\]
and we get
\begin{multline*}
\sum_{x\in S_{(0)}} \sharp_n^*(\Sp_x M)=\sum_{x\in S_{(0)}} \sharp_n^*((f_x)_*i_x^* g_! \Z_X)\\
=\sum_{x\in S_{(0)}} \sharp_n^*((f_x)_!i_x^* g_! \Z_X)
=\sum_{x\in S_{(0)}} \sharp_n^*((f_xg_x)_! \Z_{X_x})
\end{multline*}
where, as above
\[\sharp_n^*((f_xg_x)_! \Z_{X_x}) =  |X_x(\F_{p^n})|.\]

Now it is clear that $|X(\F_{p^n})|=\sum_{x\in S_{(0)}} |X_x(\F_{p^n})|$.

From 1), one can derive Proposition \ref{p3.2} for those $M$'s which belong to the smallest triangulated subcategory containing the said $g_!\Z_X$ and stable under Tate twists. By Proposition \ref{p3.3}, its pseudo-abelian envelope equals $\D(S)$; but then we run into the problem of idempotents mentioned before the proof of Theorem \ref{p1}. This leads us to the second proof:

2) Apply the $l$-adic realisation functor $R_l$.
\begin{multline*}
\sum_{x\in S_{(0)}} \sharp_n^*(\Sp_x M) = \sum_{x\in S_{(0)}} \sharp_n^*(R_l(\Sp_x M)) =  \sum_{x\in S_{(0)}} \sharp_n^*(\Sp_x  R_l(M)) \\
= \sharp_n^*(Rf_!R_l(M)) =\sharp_n^*(R_l(f_!M)) =\sharp_n^*(f_!M)
\end{multline*} 
where we used the commutation of $R_l$ with $f_!$, $i_x^*$ and duality, plus \cite[equation (2) p. 470]{SGA5}.
\end{proof}

\begin{cor}\label{c3.1}  If $S(\F_{q^n})=\emptyset$, then $\sharp_n^*(f_!M)=0$ for any $M\in \D(S)$.
\end{cor}

\begin{proof} This follows from Lemma \ref{l3.1} and Proposition \ref{p3.2}.
\end{proof}

\begin{qns} 1) Conversely, Corollary \ref{c3.1} implies Proposition \ref{p3.2} by the localisation exact triangle \eqref{eq3.1}, even for $S=\P^1-\P^1(\F_{q^n})$ by the dévissages of \cite[XV.3.2]{SGA5}. Can one find a proof which avoids the $l$-adic realisation and the trace formula of \cite{SGA5}?\\
2) The statement breaks down for $n=0$, because the right hand side is generally infinite (\emph{e.g.} for $M=\Z_S$). I don't know a formula for $\chi(f_! M)$, which appears in the functional equation just below. Can one use some renormalisation trick?
\end{qns}

\begin{thm} \label{t3.1} a) One has $\zeta(M,s)=\zeta(f_!M,s)$. In particular, there exists $Z(M,t)\in \Q(t)$ such that $\zeta(M,s)=Z(M,q^{-s})$; the zeroes and poles of $Z(M,T)$ are Weil $q$-numbers.\\
b) If $M\in d^!_{\le n} \D(S)$ (see Proposition \ref{p3.3}), these Weil $q$-numbers are effective of weights $\le 2(n+d)$, where $d=\dim S$.\\ 
c) If $S$ is projective, one has the functional equation
\[\zeta(D_S(M),-s)=(-q^{-s})^{\chi(f_!M)}\det(F_{f_!M})^{-1}\zeta(M,s)\]
where $D_S(M)=\uHom(M,f^!\Z)$. If $S$ is moreover smooth of dimension $d$, one has the functional equation
\[\zeta(M^*,d-s)=(-q^{-s})^{\chi(f_!M)}\det(F_{f_!M})^{-1}\zeta(M,s)\]
with $M^*:=\uHom(M,\Z)$.
\end{thm}

\begin{proof} a) follows from Proposition \ref{p3.2} and \eqref{eq3.3}. b) follows from Propositions \ref{p3.3} and  \ref{p3}. For c), we have the isomorphism $f_!D_S(M)\simeq (f_*M)^*\simeq (f_!M)^*$ by \cite[Th. 2.3.75 and Scholie 1.4.2 4]{ayoub}, hence
\begin{multline*}
\zeta(D_S(M),-s)= \zeta((f_!M)^*,-s) = (-q^{-s})^{\chi(f_!M)}\det(F_{f_!M})\zeta(f_!M,s)\\
= (-q^{-s})^{\chi(f_!M)}\det(F_{f_!M})\zeta(M,s)
\end{multline*}

The smooth case follows, since $f^!\simeq f^*(d)[2d]$ then (the adjoint identity to the one of \eqref{eq3.4}).
\end{proof}

\begin{rks} 1) Let $\D^\eff(S)$ be as in the proof of Proposition \ref{p3.1}. If $M\in d_{\le n} \D^\eff(S)$, Theorem \ref{t3.1} a) may be refined: the zeroes and poles of $\zeta(M,s)$ are effective Weil numbers of weights $\le 2n$.\\
2) Theorem \ref{t3.1} b) extends to $S$ proper by \cite{cis-deg}.\\
3) I don't know a formula for $\zeta(D_S(M),-s)$ when $S$ is not proper.
\end{rks}

\subsection{The general case} We come back to the situation where $f:S\to \Spec \Z$ is an arbitrary $\Z$-scheme of finite type. We write $d$ for its relative dimension, \emph{i.e} the maximal dimension of its closed fibres.

\begin{thm}\label{t3.3} 
Let $f:S\to T$ be a morphism of $\Z$-schemes of finite type. Then $\zeta(M,s) = \zeta(f_!M,s)$ for any $M\in \D(S)$ (as formal Dirichlet series).
\end{thm}

In Corollary \ref{c6.1} below, this Dirichlet series will be shown to be convergent.

\begin{proof} We immediately reduce to the case $S=\Spec \Z$.  For a prime number $p$, let $S_p$ be the fibre of $f$ at $p$; we have a Cartesian square similar to \eqref{eq3.5}:
\begin{equation}\label{eq3.6}
\begin{CD}
S_p@>I_p>> S\\
@Vf_p VV @Vf VV\\
\Spec \F_p@>i_p>> \Spec \Z.
\end{CD}
\end{equation}

 The equality
\[\zeta(M,s) = \prod_p \zeta(I_p^*M,s)\]
follows from the definition. 

 By Theorem \ref{t3.1} a) and proper base change, we have 
\[\zeta(I_p^*M,s)=\zeta((f_p)_!I_p^*M,s)=\zeta(i_p^*f_!M,s)\]
hence
\[\zeta(M,s) = \prod_p \zeta(i_p^*f_!M,s)=\zeta(f_!M,s)\]
again by definition.
\end{proof}

\section{Motives over $\R$ and $\C$}

\subsection{Hodge structures} Set
\[\Gamma_\R(s) = \pi^{-s/2}\Gamma(s/2),\qquad \Gamma_\C(s) = 2(2\pi)^{-s}\Gamma(s)\]
where $\Gamma(s)$ is Euler's Gamma function. (For convenience, we take for $\Gamma_\C(s)$ twice the function in Serre \cite[\S 3]{serre}.)

If $V$ is a (pure) complex Hodge structure (i.e., a finite-dimensional $\C$-vector space
provided with a decomposition $V=\bigoplus_{(p,q)\in \Z\times\Z} V^{p,q}$), one defines
\[\Gamma(V,s) = \prod_{(p,q)} \Gamma_\C(s-\inf(p,q))^{h(p,q)}\]
with $h(p,q)=\dim V^{p,q}$.

If $V$ is a (pure) real Hodge structure (i.e. a complex Hodge structure plus an involution
$\sigma$ such that $\sigma V^{p,q} = V^{q,p}$), one defines
\[\Gamma(V,s) = \prod_n \Gamma_\R(s-n)^{h(n,+)}\Gamma_\R(s-n+1)^{h(n,-)}\prod_{p<q}
\Gamma_\C(s-p)^{h(p,q)}\] where $h(p,q)$ is as above and $h(n,\epsilon) = \dim(V^{n,n}\mid \sigma =
(-1)^n\epsilon)$.

\subsection{Pure motives} If $k=\R$ or $\C$, we have a realisation functor
$H:\sM_\hom(k,\Q)\to \Hodge_k$, where $\Hodge_k$ is the category of pure $k$-Hodge structures.
For $M\in \sM_\hom(k)$, we define
\[\Gamma(M,s) = \Gamma(H(M)^*,s).\]

\subsection{Triangulated motives} Using Bondarko's isomorphism 
\[K_0(\Chow(k,\Q))\iso K_0(\DM_\gm(k,\Q))\] 
we may extend the above definition to all objects of $\DM_\gm(k,\Q)$ (alternately, we could go
through the Hodge realisation on $\DM_\gm(k,\Q)$).

The identities of Corollary \ref{c1} c) hold for $\Gamma(M,s)$.

\begin{rk} One should compare this definition with the much more sophisticated one for mixed motives in \cite[III.1]{fpr}, using \emph{mixed} Hodge structures.
\end{rk}

\section{Two elementary lemmas on Dirichlet series}

\begin{lemma}\label{l1} Let $(R_p)$ be a sequence of rational functions with complex
coefficients, indexed by the prime numbers. We assume that $R_p(0)=1$ for all $p$ and:
\begin{thlist}
\item There is an integer $w$ such that, for almost all $p$, the inverse zeroes and poles of
$R_p$ have absolute value $\le p^{w/2}$.\\
\item The heights of the $R_p$ are bounded independently of $p$ (N.B.: here, the \emph{height}
of a rational function $R=P/Q$ is $\deg(P)+\deg(Q)$ if $P$ and $Q$ are coprime polynomials).
\end{thlist}
Let $L(s)=\prod_p R_p(p^{-s})$. Then $L(s)$ is a Dirichlet series with absolute convergence abscissa $\le w/2+1$.
\end{lemma}

\begin{proof} Let $\lambda$ be an inverse pole of $R_p$. Then, for $s=\sigma + it\in \C$:
\[\left|(1-\lambda p^{-s})^{-1}\right|=\left|\sum_{n= 0}^\infty \lambda^n p^{-ns}\right|\le \sum_{n= 0}^\infty p^{-n(\sigma-w/2)}\]
converging as soon as $\sigma > w/2$.
If now $\lambda$ is an inverse zero of $R_p$, we have
\[\left|1-\lambda p^{-s}\right|\le 1+ |p^{w/2-s}|\le \sum_{n= 0}^\infty p^{-n(\sigma-w/2)}.\]

Thus, if the height of $R_p$ is $\le H$, we have
\[\left|R_p(p^{-s})\right|\le \left(\sum_{n= 0}^\infty p^{-n(\sigma-w/2)}\right)^H=\left(1-p^{-(\sigma-w/2)}\right)^{-H}.\]

Collecting, we find
\[\left|L(s)\right|\le
\left(\prod_p\left(1-p^{-(\sigma-w/2)}\right)^{-1}\right)^H=\zeta(\sigma-w/2)^H\] which
converges for $\sigma-w/2>1$, as is well-known.
\end{proof}

\begin{lemma}\label{l2} a) Let $f=\sum_{n= 1}^\infty a_n n^{-s}$ be a convergent Dirichlet series
with complex coefficients, with $a_1=1$. Then the equation $f(s)= g(s)/g(s+1)$ has a unique
solution as a convergent Dirichlet series, namely
\[g(s) = \prod_{m= 0}^\infty f(s+m).\]
Moreover, $g$ has the same absolute convergence abscissa as $f$.\\
b) If the coefficients of $f$ belong to a subring  $R$ of $\C$, so do those of $g$. 
\end{lemma}

\begin{proof} a) \emph{Uniqueness:} if $g_1,g_2$ are two solutions, then $h=g_1/g_2$ verifies
\[h(s)=h(s+1).\]

If $h(s)=\sum c_n n^{-s}$, this gives the identity
\[u(s)=\sum(c_n-c_n/n) n^{-s} = 0.\]

Since $g_1,g_2$ are convergent, so are $h$ and $u$, and it is well-known that this implies
$c_n - c_n/n=0$ for all $n$, hence $c_n=0$ for all $n>1$.

\emph{Existence:} Let us check first that $g(s)$ converges as a formal Dirichlet series.
Indeed:
\[g(s) = \prod_{m= 0}^\infty \left(\sum_{n=1}^\infty \frac{a_n}{n^m} n^{-s}\right).\]

In this product, the coefficient $b_n$ of $n^{-s}$ is
\[b_n = \sum_{r_1\dots r_k=n} \sum_{m_1\ge 0,\dots,m_k\ge 0} \frac{a_{r_1}}{r_1^{m_1}}\dots
\frac{a_{r_k}}{r_k^{m_k}}=\sum_{\substack{r_1\dots r_k=n\\r_1,\dots,r_k>1}}
\frac{a_{r_1}}{1-r_1^{-1}}\dots \frac{a_{r_k}}{1-r_k^{-1}} .\]

 Suppose that $f(s)$ converges absolutely for $\Re(s)>c$. Then $|a_n|=o(n^{c+\epsilon})$ for
all $\epsilon>0$. Therefore
\begin{multline*}
|b_n]\le \left|\sum_{\substack{r_1\dots r_k=n\\r_1,\dots,r_k>1}}
\frac{a_{r_1}}{1-r_1^{-1}}\dots \frac{a_{r_k}}{1-r_k^{-1}}\right|\\ = o(n^{c+\epsilon})
\sum_{\substack{r_1\dots r_k=n\\r_1,\dots,r_k>1}} \frac{1}{1-r_1^{-1}}\dots
\frac{1}{1-r_k^{-1}}=o(n^{c+\epsilon})g^0_n
\end{multline*}
where  $g^0_n$ is the $n$-th coefficient of
\[g^0(s) =\zeta(s)\zeta(s+1)\dots\]

 To study the absolute convergence of this product, we look at the log of the corresponding
Eulerian product, for $s\in\R$:
\begin{multline*}
 \log g^0(s)=\sum_p\sum_{m=0}^\infty -\log(1-p^{-s-m})=\sum_p\sum_{m=0}^\infty\sum_{k=1}^\infty
\frac{p^{-k(s+m)}}{k}\\ =\sum_p \sum_{k=1}^\infty \frac{p^{-ks}}{k}\frac{1}{1-p^{-k}}\le
2\sum_p \sum_{k=1}^\infty \frac{p^{-ks}}{k} \\ = 2\sum_p -\log(1-p^{-s}) = \log(\zeta(s)^2).
\end{multline*}

The $n$-th coefficient of $\zeta(s)^2$ is the number of divisors of $n$, which is
$o(n^{\epsilon})$ for all $\epsilon>0$. Hence $|b_n|=o(n^{c+\epsilon})$ for all $\epsilon >0$.
It follows that $g(s)$ converges absolutely for $\Re(s)>c$. But conversely, if $g(s)$ converges
absolutely for $\Re(s)>c$, so does $f(s)=g(s)/g(s+1)$.

b) is obvious from the formula giving $g(s)$.

\end{proof}

\section{Motives over a ring of integers}

Let $K$ be a number field, with ring of integers $O_K$. Consider the category $\D(O_K)$. For each place $v$ of $K$, we have a homomorphism $\phi_v:O_K\to \kappa(v)$, where $\kappa(v)$ is
\begin{itemize}
\item The residue field at $v$ if $v$ is finite;
\item The completion of $K$ at $v$ if $v$ is archimediean.
\end{itemize}

In each case, we have a pull-back functor
\[\phi_v^*:\D(O_K)\to \D(\kappa(v)).\]

We shall also use the fact that $\Z_{\Spec O_K}$ is a dualising object, of $\D(O_K)$ which follows from \cite[Th. 2.3.73]{ayoub} and \cite[Cor. 5.15]{dJ2}. We write $M\mapsto M^*$ for the corresponding duality functor.

\begin{thm} \label{t6.2} Let $M\in \D(O_K)$. Then there exists a nonempty open subset $U\subseteq \Spec O_K$ such that the cohomology sheaves $H_l^i(M)$ are locally constant constructible for any prime number $l$ invertible on $U$.
\end{thm}

\begin{proof} We are in a case where $\D(O_K)\simeq \DM_\gm(O_K,\Q)$, so we can reason as in the proof of Lemma \ref{l2.2a}; this reduces us to the case where $M=M(X)$ with $f:X\to \Spec O_K$ smooth. Then the result is a special case of results of Illusie \cite[Th. 2.1]{illusie}. (Recall the proof: by Hironaka's resolution of singularities and some spread-out, there exists $U$ and an open immersion $j:X_U\inj \bar X$, where $\bar f:\bar X\to U$ is smooth projective and the closed complement $\bar X-X_U$ is the support of a divisor $D$ with strict normal crossings, relative to $U$. By \cite[Arcata V.3.1]{SGA412}, the sheaves $R^i\bar f_*\Q_l$ are locally constant, and so are the corresponding sheaves for all intersections of the components of $D$. The Leray spectral sequence for $j$ and cohomological purity then show that the same holds for the $R^i(f_{|U})_*\Q_l$,  \emph{cf.} \cite[Lemme 3.1]{illusie}.)
\end{proof}

\begin{cor}\label{l3} For any $M\in \D(O_K[1/l])$, the
function
\[\fp\mapsto b_\fp(M)= \sum_{i\in\Z}\dim H^i_l(M_\fp)\]
from primes of $O_K[1/l]$ to $\N$ is bounded. \qed
\end{cor}

\begin{thm}\label{t6.1} The Dirichlet series $\zeta(M,s)$ has a finite convergence abscissa. More precisely, if
$M^*\in d_{\le n} \D(O_K)$, then $\zeta(M,s)$ converges absolutely for $\Re(s)>n+1$.
\end{thm}

\begin{proof} Let $M^*\in d_{\le n} \D(O_K)$. By Proposition \ref{p3} and Corollary
\ref{l3}, the hypotheses of Lemma \ref{l1} are satisfied for $\zeta(M,s)$, with $w=2n$. This
proves the statement for effective motives, hence in general.
\end{proof}

\begin{cor}\label{c6.1} For any $\Z$-scheme of finite type $S$ and any $M\in \D(S)$, the formal Dirichlet series $\zeta(M,s)$ has a finite convergence abscissa.
\end{cor}

\begin{proof} This follows from Theorems \ref{t3.3} and \ref{t6.1} a).
\end{proof}

\begin{rk} Denis-Charles Cisinski pointed out that there are more motivic and uniform methods to obtain  bounds as in Corollary \ref{l3} (\cite[Th. 6.3.26]{cis-deg2}, \cite{schol-etal}, \cite[Th. 2.4.2]{cisinski}). But all these theorems rest \emph{in fine} on the smooth and proper base change theorem for étale cohomology \cite[Arcata, V, Th. 3.1]{SGA412}, so they do not seem to bring something essentially new here as the existence of a bound is sufficient for Theorem \ref{t6.1} and its corollary.
\end{rk}

\begin{prop}\label{p6.1} For $M\in \D(O_K)$, let
\[\quad \xi(M,s)= |d_K|^{s\chi(M)/2}\prod_{v\mid\infty}\Gamma(\phi_v^* M,s)\cdot \zeta(M,s) \]
where $d_K$ is the absolute discriminant of $K$. Then  $\xi(\pi_* M,s)=\xi(M,s)$, where $\pi_*$ is the push-forward functor $\D(O_K)\to \D(\Z)$. 
\end{prop}

\begin{proof} It suffices to check this for the Gamma-discriminant part, and it follows from an elementary computation.
\end{proof}

To be complete, we define the completed $\zeta$ function of a motive over a $\Z$-scheme of finite type.

\begin{defn}\label{d6.1}
If $M\in \D(S)$ where $f:S\to \Spec \Z$ is a $\Z$-scheme of finite type, we set $\xi(M,s)=\xi(f_!M,s)$.
\end{defn}

Note that $\xi(M,s)=\zeta(M,s)$ if $f$ is not dominant, and Theorem \ref{t3.3} still holds when replacing $\zeta$ by $\xi$. By Proposition \ref{p6.1}, this definition does extend the case $S=\Spec O_K$.

\section{A theorem of Serre}

For $M\in \D(O_K)$ and $\fp\subset O_K$, define
\[N_M(\fp) = \sharp(M_\fp)\]
the number of points of $M$ modulo $\fp$.

\begin{thm}\label{t3} Let $M\in \D(O_K)$. Suppose that $\zeta(M,s)$ is
not a finite product of Euler factors. Then the set
\[\{\fp\mid N_M(\fp) = 0\}\]
has a density $1-\epsilon$, with
\[\epsilon \ge \frac{1}{b_\infty(M)^2}\]
where $b_\infty(M)=\sum_i \dim H^i_l(M_K)$.
\end{thm}

\begin{proof} It is the same as Serre's \cite[Th. 6.17]{NX}, which is the special case
$M=M^{BM}(X)\oplus M^{BM}(X')[1]$ for $X,X'$ $O_K$-schemes of finite type. For $U$ as in Theorem \ref{t6.2}, one may compute
the traces of the geometric Frobenius $F_{M_\fp}$ acting on $H_l^*(M_\fp)$, for $\fp\in U$, as
the traces of the inverse of the [conjugacy class of the] arithmetic Frobenius $\phi_\fp\in
Gal(\bar K/K)$ acting on
$H_l^*(M_K)$. The statement then reduces to the following theorem of Serre \cite[Th. 5.15]{NX}:

\begin{thm}\label{tserre} Let $G$ be a compact group, $K$ be a locally compact field of
characteristic
$0$ and let $\rho:G\to GL_n(K)$, $\rho':G\to GL_{n'}(K)$ be two continous $K$-linear
representations  of
$G$. Then
\begin{thlist}
\item either $\tr \rho= \tr \rho'$;
\item or the set $\{g\in G\mid \tr(\rho)(g)\ne \tr(\rho')(g)\}$ has a Haar density $\ge
\frac{1}{(n+n')\sup(n,n')}$.
\end{thlist}
\end{thm}

To apply Serre's theorem, we represent $G=G_K$ on $H^{even}_l(M_K)$, yielding $\rho$, and on
$H^{odd}_l(M_K)$, yielding $\rho'$. Here, $K=\Q_l$. If both those vector spaces are $0$, then
$\sH_l(M)$ is supported on $\Spec O_K[1/l]-U$, where $U$ is the open set as above; then all
Euler factors of $\zeta(M,s)$ at $\fp\in U$ are equal to $1$. Otherwise, we may apply the
theorem. In case (i), we get the same conclusion on the Euler factors as above. Case (ii)
yields the conclusion of Theorem \ref{t3}.
\end{proof}

More generally, a large part of Serre's results in \cite{NX} seem to extend to objects of $\D(O_K)$ without difficulty. 

\section{Some six functors algebra}

\subsection{$K_0$ of triangulated categories}\label{s8.A} Let
\[0 \to\sT'\by{i} \sT\by{p} \sT''\to 0\]
be a short exact sequence of triangulated categories: $i$ is a thick embedding and $\sT/\sT'\iso\sT''$. We assume that we are in the Verdier situation: $i$ has a right adjoint $\pi$, hence $p$ has a right adjoint $j$ and any object $M\in \sT$ fits in a functorial exact triangle
\[i\pi M\to M\to pj M\by{+1}.\]

From this, it follows:

\begin{lemma} In the above situation, the map
\[K_0(\sT)\by{\left(\begin{smallmatrix}\pi\\ p \end{smallmatrix}\right)} K_0(\sT')\oplus K_0(\sT'')\]
is an isomorphism. Alternately, we have an exact sequence
\[0\to K_0(\sT')\by{i} K_0(\sT)\by{p} K_0(\sT'')\to 0\]
split by $\pi$ and $j$.
\end{lemma}

\begin{proof} Since $i$ and $j$ are fully faithful, the first map has the right inverse $(i,j)$. It is also a left inverse thanks to the exact triangle above.
\end{proof}

Let $(\sT_\alpha)_{\alpha\in A}$ be an filtered inductive system of triangulated categories, and let $\sT=\colim_\alpha \sT_\alpha$. We have an induced homomorphism
\begin{equation}\label{eq8.7}
\colim_\alpha K_0(\sT_\alpha)\to K_0(\sT).
\end{equation}

\begin{lemma}\label{l8.1} \eqref{eq8.7} is an isomorphism.
\end{lemma}

\begin{proof} Write $i_\alpha:\sT_\alpha\to \sT$ for the canonical functor. Let $X\in \sT$. Then $X\simeq i_\alpha X_\alpha$ for some $(\alpha,X_\alpha)$. This shows that \eqref{eq8.7} is surjective.

Let $\alpha\in A$ and  $x\in K_0(\sT_\alpha)$ be such that $i_\alpha x=0$. Writing $x=\sum_{i\in I} n_i[X_i]$ for $I$ finite, $n_i\in\Z$ and  $X_i\in \sT_\alpha$, the hypothesis means that we have an equality
\[ \sum_{i\in I} n_i[i_\alpha X_i] = \sum_{j\in J} m_j([Y_j]-[Y'_j]-[Y''_j])\]
in the free group with generators the isomorphism classes of objects of $\sT$, where $J$ is finite, $m_j\in\Z$ and $Y'_j\to Y_j\to Y''_j\by{+1}$ are exact triangles. All these exact triangles come from exact triangles in $\sT_\beta$ for some $\beta$ dominating $\alpha$. This shows that \eqref{eq8.7} is injective.
\end{proof}

For simplicity, we set
\[K_0^\sM(S) = K_0(\D(S))\]
for any scheme $S$.

Let $K$ be a global field. We write $C$ for $\Spec O_K$ if $K$ is  number field or for the smooth projective model of $K$ in positive characteristic.

Let $j:U\subseteq V$ be two nonempty open subsets of $C$, and let $i:Z\to V$ be to complementary closed immersion. We have an exact sequence of triangulated categories
\begin{equation}\label{eq9.4}
0\to \D(Z)\by{i_*} \D(V)\by{j^*}\D(U)\to 0
\end{equation}
which is split as in the previous section by $i^!$ and $j_*$. Note also that
\[\D(Z)=\coprod_{v\in Z} \D(\kappa(v)).\]

Hence we get a short exact sequence
\[0\to \bigoplus_{v\in Z} K_0^\sM(\kappa(v)) \by{((i_v)_*)} K_0^\sM(V)\by{j^*} K_0^\sM(U)\to 0\]
which is split, although we shall not use this.

By Ivorra \cite[Prop. 4.16]{ivorra}, we have an equivalence
\[2-\colim_U \D(U) \iso \D(K)\]
hence, by Lemma \ref{l8.1}, a (non split) short exact sequence
\begin{equation}\label{eq8.0}
0\to \bigoplus_v K_0^\sM(\kappa(v)) \by{((i_v)_*)} K_0^\sM(C)\by{j^*} K_0^\sM(K)\to 0
\end{equation}
where $v$ runs through all the closed points of $C$.
\subsection{A purity theorem} 

For the sequel, we note that the theory $\D$ verifies the axioms of \cite[Def. 2.3.1]{ayoub} by \emph{loc. cit.}, Prop. 4.5.31.

Let $S$ be a scheme, $i:Z\to S$ a closed immersion and $j:U\to S$ the complementary open immersion. For $M,N\in \D(S)$, we have a natural transformation \cite[\S 2.3.2]{ayoub}
\begin{equation}\label{eq8.1}
r_{M,N}:i^*M\otimes i^!N\to i^!(M\otimes N)
\end{equation}
which is not an isomorphism in general (\emph{loc. cit.}, Remark 2.3.13). However,

\begin{thm}\label{t8.1} If $M$ is strongly dualisable, \eqref{eq8.1} is an isomorphism.
\end{thm}

\begin{proof}[Proof (J. Ayoub)]  Let $j:U\to S$ be the complementary open immersion. By the localisation exact triangle
\begin{equation}\label{eq8.6}
i_*i^!N\to N\to j_*j^*N\by{+1} 
\end{equation}
of \cite[\S 1.4.4]{ayoub}, we reduce to the cases where $N$ is of the form $j_* N'$ or $i_*N''$.

In the first case, the left hand side is $0$, and so is the right hand side by the projection formula
\[M\otimes j_* N' \simeq j_*(j^*M \otimes N')\]
which holds because $M$ is dualisable (another lemma of Ayoub, \emph{cf.} \cite[Lemma 9.3.1]{ffihes}).

In the second case, \eqref{eq8.1} is the composition 
\[i^*M \otimes i^!i_*N'' \simeq i^!i_*(i^*M \otimes i^!i_*N'') \simeq
                                       i^!(M \otimes i_*i^!i_*N'') \to i^!(M\otimes i_*N'')\]
by the strong monoidality of $i_*$ \cite[Lemma 2.3.6]{ayoub}, and the last map is invertible: the counit $i_*i^! \to Id$ is invertible when applied to an object of the form $i_*N''$.
\end{proof}

Let $M\in \D(S)$. Applying $i^*$ to \eqref{eq8.6}, we get another exact triangle
\[i^!M\by{\tau_M} i^*M\to i^*j_*j^*M\by{+1}. \]

\begin{cor}\label{c8.1} Suppose that $\tau_{\Z_S}=0$. Then $\tau_M=0$ for any strongly dualisable $M$.
\end{cor}

\begin{proof}[Proof (J. Ayoub)]  We have a commutative diagram
\[\begin{CD}
i^*M \otimes i^!N @>r_{M,N}>> i^!(M\otimes N) \\
@V{1\otimes \tau_N}VV @V{\tau_{M\otimes N}}VV\\
i^*M \otimes i^*N @>\sim >> i^*(M\otimes N)
\end{CD}\]
where the top map is an isomorphism by Theorem \ref{t8.1}. The conclusion follows by taking $N=\Z_S$.
\end{proof}

\begin{rk}\label{r8.1} The hypothesis of Corollary \ref{c8.1} is verified in particular when $Z=\Spec \kappa$ is a closed point of $S$: then $\tau_{\Z_S}\in \D(\kappa)(\Z_Z(-d)[-2d],\Z_Z)=H^{2d}(\kappa,\Q(d))=0$.
\end{rk}

\begin{cor} \label{c8.2} Suppose that $S$ and $Z$ are regular and that $i$ is of pure codimension $c$. Then the morphism \eqref{eq8.1} taken with $N=\Z_S$ induces an isomorphism 
\[i^*M(-c)[-2c]\iso i^!M\]
for any dualisable $M$.
\end{cor}

\begin{proof} This follows from Theorem \ref{t8.1} and \cite[Th. 14.4.1]{cis-deg} in the language of Beilinson motives, or from \cite[Cor. 7.5]{ayoubetale} in the language of $\DA^\et$.
\end{proof}

\begin{prop}\label{p8.1} Let $\D^\proj(S)$ be the thick subcategory of $\D(S)$ generated by the $M(X)(n)$ with $X$ smooth projective over $S$ and $n\in \Z$. Suppose that $\Z_S$ is a dualising object of $\D(S)$. Then all objects of $\D^\proj(S)$ are strongly dualisable. In particular, Theorem \ref{t8.1}, Corollary \ref{c8.1} and Corollary \ref{c8.2} apply to them.
\end{prop}

\begin{proof} The argument is the same as for \cite[Th. 2.2]{riou}.
\end{proof}

\begin{rk} As Ayoub pointed out, if one wants to prove Corollary \ref{c8.2} for $M=f_\#\Z_X$ with $f$ smooth projective, one can avoid using Theorem \ref{t8.1} by the following direct computation: $M\simeq f_!f^!\Z_S$ and $i^!f_!f^!\simeq (f_Z)_!f_Z^!i^!$ by the base change isomorphisms, where $f_Z$ is the pull-back of $f$ along $i$.
\end{rk}

\section{Zeta and $L$-functions of motives over a global field}\label{s9}

Let $K$ be a global field. We keep the notation of \S \ref{s8.A}: $C$ denotes either $\Spec O_K$ when $K$ is a number field and $O_K$ is its ring of integers, or the smooth projective model of $K$ over its field of constants $\F_q$ when $K$ is of positive characteristic.

\subsection{Zeta functions up to finite Euler products} 

Let as above $\Dir$ denote the group of convergent Dirichlet series, and let $\Eul$ be the subgroup generated by Euler factors, i.e. Dirichlet series of the form $R(p^{-s})$, where $R\in\Q(t)$ and $p$ is a prime number. From this exact sequence we deduce a map
\[\bar\zeta:K_0^\sM(K)\to \Dir/\Eul\]
induced by $DM_\gm(O_K,\Q)\ni M\mapsto \zeta(M,s)$.

Let $X$ be a smooth projective $K$-variety, and let $\Sigma_K(X)$ be the set of finite places of $K$ where $X$ has good reduction. Recall that, in \cite[Prop. 5.6 and Th. 5.7]{zetaL}, we defined an ``approximate zeta function'' by the formula
\[\zeta_{appr}(X,s)=\prod_{v\in \Sigma_K(X)} \zeta(X(v),s)\]
where $X(v)$ is the special fibre of a smooth model of $X$ over $\Spec O_v$. (The point is that $\zeta(X(v),s)$ does not depend on the choice of $X(v)$.) Then the following is obvious by construction:

\begin{prop} We have $\zeta_{appr}(X,s)=\bar \zeta(X,s)$
in $\Dir/\Eul$.\qed
\end{prop}

\subsection{A ``total" $L$-function for motives over $\Spec K$}\label{total}

We go back to the short exact sequence of triangulated categories
\[0\to \coprod_v \D(\kappa(v))\by{(i_v)_*} \D(C)\by{j^*}\D(K)\to 0\]
which is the $2$-colimit of the exact sequences \eqref{eq9.4}. It sits fully faithfully into a short exact sequence of larger categories
\[0\to \coprod_v \DA^\et(\kappa(v),\Q)\by{(i_v)_*} \DA^\et(C,\Q)\by{j^*}\DA^\et(K,\Q)\to 0.\]

In this sequence, $j^*$ has the right adjoint $j_*$ for Brown representability reasons. If $M\in \D(K)$, $j_* M$ is of course not constructible in general; nevertheless we would like to define a ``total $L$-function" of $M$ by the formula
\[L^\tot(M,s)=\zeta(j_*M,s).\]

Sense can be made of this formula as follows:

We may write $M=j_U^*M_U$, where $M_U\in \D(U)$, for some open subset $U$ of $C$. Write $j=j'_U j_U$. Then $j_*M = (j'_U)_*(j_U)_*j_U^*M_U$. Let $j_{U,V}:V\to U$ be an open subset of $U$. From the exact triangles
\begin{equation}\label{eq8.2}
\bigoplus_{v\in U-V} (i_v)_*i_v^! M_U \to M_U\to (j_{U,V})_*j_{U,V}^*M_U\by{+1}
\end{equation}
we deduce in the (co)limit an exact triangle
\[\bigoplus_{v\in U} (i_v)_*i_v^! M_U \to M_U\to (j_{U})_*j_{U}^*M_U\by{+1}\footnote{For lightness of notation, we write $v\in U$ rather than $v\in U_{(0)}$, and similarly in the sequel.}\]
hence, after applying $(j'_U)_*$, an exact triangle
\begin{equation}\label{eq8.3}
\bigoplus_{v\in U} (i_v)_*i_v^! M_U \to (j'_U)_*M_U\to j_*M\by{+1}.
\end{equation}

Note that the latter may also be written
\[\bigoplus_{v\in C} (i_v)_*i_v^! (j'_U)_*M_U \to (j'_U)_*M_U\to j_*M\by{+1}\]
as one sees for example by applying the localisation exact triangles to $(j'_U)_*M_U$. (Here we abuse notation by identifying the closed immersions $v\inj U$ and $v\inj C$, with the common name $i_v$.) The
motive $(i_v)_*i_v^! (j'_U)_*M_U$ is $0$ if $v\notin U$, because then $i_v^!
(j'_U)_*=0$.

\begin{prop}\label{p8.2} 
For any closed point $v$ of $C$, let $O_v$ be the local ring of $v$ and $j_v:\Spec K\inj\Spec O_v$
the corresponding open immersion. Then, with the above notation, we have the relation
\[ [i_v^*(j'_U)_*M_U]-[i_v^! (j'_U)_*M_U]= [i_v^*(j_v)_*M]\in K_0^\sM(\kappa(v)).\]
In particular, the left hand side does not depend on the choice of $U$ and $M_U$, and is
triangulated in $M$.
\end{prop}

\begin{proof} Let $(V,M_V)$ be another model of $M$. Then $M_U$ and $M_V$ become isomorphic
after restricting to some open subset of $U\cap V$, so we may assume $V\subseteq U$ and $M_V =
j_{U,V}^*M_U$. Hence, by \eqref{eq8.2}, it suffices to show that
\[
i_w^*(i_v)_*= i_w^!(i_v)_*= 
\begin{cases}
0& \text{if $w\ne v$}\\
Id_{\D(\kappa(v))}& \text{if $w\ne v$.}
\end{cases}
\]

Both formulas are obvious, the second because $(i_v)_*$ is fully faithful. This shows that
the left hand side of the formula only depends on $M$.

Since $j^*(j'_U)_*M=M$, we may now suppose that $U=C$ in \eqref{eq8.3}.   Let $j'_v:\Spec
O_v\inj C$ be the other inclusion. Applying $(j'_v)^*$ to  \eqref{eq8.3}, we get an exact
triangle
\[(i_v)_*i_v^! M_C \to (j'_v)^*M_C\to (j_v)_*M\by{+1}\]
hence the formula.
\end{proof}

\begin{defn}\label{l9.1} For any $M\in \D(K)$ and any $v\in C_{(0)}$, we set
\[L^\tot_v(M,s)= \zeta(i_v^*(j_v)_*M,s).\]
\end{defn}

\begin{defn} Let $M\in \D(K)$ and let $v\in C_{(0)}$. We say that $M$ \emph{has good
reduction at $v$} if $M=j_v^*\sM$, with $\sM\in \D^\proj(O_v)$ (see Proposition \ref{p8.1}). We say that such
an $\sM$ is a \emph{good model} of $M$ at $v$.
\end{defn}

Let $M\in \D(K)$, and let $\sM\in \D(O_v)$ be such that $j_v^*\sM=M$.  The
exact triangle
\[(i_v)_*i_v^!\sM\to \sM\to (j_v)_*M\by{+1}\]
gives, after applying $i_v^*$, an exact triangle
\[i_v^!\sM\to i_v^*\sM\to i_v^*(j_v)_*M\by{+1}.\]

If $M$ has good reduction at $v$ and $\sM$ is a good model, this triangle reads as
\[i_v^*\sM(-1)[-2]\to i_v^*\sM\to i_v^*(j_v)_*M\by{+1}\]
thanks to Theorem \ref{t8.1} and Proposition \ref{p8.1}. (The first map is trivial by Corollary \ref{c8.1} and Remark \ref{r8.1}, although we shall not need this.) Thus
\begin{equation}\label{eq8.4}
L^\tot_v(M,s) = \frac{\zeta(i_v^*\sM,s)}{\zeta(i_v^*\sM,s+1)}
\end{equation}
in this case, thanks to Corollary \ref{c1} c).

\begin{thm}\label{t8.2} With the above notation,\\
a) The Euler product 
\[L^\tot(M,s)= \prod_{v\in C} L_v^\tot(M,s)\]
is a Dirichlet series which converges absolutely for $\Re(s)\gg 0$.\\
b) If $M=M(X)^*$, where $X$ is smooth projective, we have for $v\nmid l$
\[L^\tot_v(M,s)=\prod_{i\in\Z} L_v^\tot(H^i_l(X),s)^{(-1)^i}\] 
where $v$ runs through the maximal ideals of $O_K$ and
\[L_v^\tot(H^i_l(X),s) = \frac{\det(1-N(v)^{-s}\phi_v^{-1}\mid H^0(I_v,H^i(\bar X,\Q_l))}{\det(1-N(v)^{-s}\phi_v^{-1}\mid H^1(I_v,H^i(\bar X,\Q_l))}\]
where $\phi_v$ and $I_v$ are respectively a Frobenius at $v$ and the inertia group at $v$.\\
c) If $X$ has good reduction at $v$ in b), with special fibre $X_v$, then
\[L_v^\tot(M,s) = \frac{\zeta(X_v,s)}{\zeta(X_v,s+1)}.\]
\end{thm}

\begin{proof} 
a) In view of Theorem \ref{t6.1} and Proposition \ref{p8.2}, it suffices to see that
$i_v^!M_U\simeq i_v^*M_U(-1)[-2]$ for almost all $v$. But for $V\subseteq U$ small
enough, $j_{U,V}^*M_U\in \D^\proj(V)$; hence this follows from Theorem \ref{t8.1}.

b) We have the isomorphism
\[R^l(i_v^*(j_v)_*M)\simeq i_v^*R(j_v)_* R^l(M)\]
for $M\in \D(K)$. 

For $M=M(X)^*$, $f:X\to \Spec K$ smooth projective, we have
\[R^l(M)= Rf_*\Q_l\]
hence
\[H^i(R^l(M)) = H^{i}(\bar X,\Q_l).\]

Let $I_v$ be the absolute inertia group at $v$. For an $l$-adic sheaf $\sF$ on $\Spec K$, we have
\[i_v^*R^q(j_v)_*\sF =
\begin{cases}
H^0(I_v,\sF(\bar K)) &\text{if $q=0$}\\
H^1(I_v,\sF(\bar K)) &\text{if $q=1$}\\
0 &\text{if $q>1$}
\end{cases}\]
with Frobenius action induced by the action of $G_K$. 
If $\sF=H^i(\bar X,\Q_l)$ for $X$ smooth projective, we have
\[
L(\kappa(v),H^i(I_v,\sF(\bar K)),s)=
 \det(1-N(v)^{-s}\phi_v^{-1}\mid H^{i}(I_v,H^i(\bar X,\Q_l))).
\]
and
\[\frac{L(\kappa(v),H^0(I_v,\sF(\bar K)),s)}{L(\kappa(v),H^1(I_v,\sF(\bar K)),s)}
=\frac{\det(1-N(v)^{-s-1}\phi_v \mid H^1(I_v,H^i(\bar
X,\Q_l)))}{\det(1-N(v)^{-s-1}\phi_v \mid H^0(I_v,H^i(\bar X,\Q_l)))}.\]
\end{proof}

\subsection{The nearby $L$-function of a motive over $\Spec K$}

We now apply Lemma \ref{l2} to $L^\tot(M,s)$. This gives

\begin{defn}\label{d9.1} For $M\in \D(K)$ and a prime $v$, we define $L_v^\near(M,s)$ as the
unique Dirichlet series (with initial coefficient $1$) such that 
\[\frac{L_v^\near(M,s)}{L_v^\near(M,s+1)} = L_v^\tot(M,s).\]
We set
\[L^\near(M,s)=\prod_v L_v^\near(M,s).\]
\end{defn}


Clearly,
\[\frac{L^\near(M,s)}{L^\near(M,s+1)} = L^\tot(M,s)\]
which shows by Lemma \ref{l2} that $L^\near(M,s)$ is a convergent Dirichlet series with the
same absolute convergence abscissa as $L^\tot(M,s)$.

If $M$ has good reduction at $v$, then for any good model $\sM$ at $v$, one has
\begin{equation}\label{eq9.1}
L_v^\near(M,s)= \zeta(i_v^*\sM,s)
\end{equation}
thanks to \eqref{eq8.4}. We shall now handle the general case and relate $L^\near(M,s)$ with
Ayoub's nearby cycle functor $\Upsilon$ \cite{ayoubetale}.

\begin{thm}\label{t9.1} We have $L^\near_v(M,s)=\zeta(\Upsilon_v M,s)$, where $\Upsilon_v:\D(K)\allowbreak\to \D(\kappa(v))$ is the ``unipotent" specialisation functor associated to $O_v$ as in \cite[Th. 11.13]{ayoubetale}.
\end{thm}

\begin{proof} By \emph{loc. cit.}, Th. 11.16, there is an exact triangle
\begin{equation}\label{eq9.2}
i_v^* (j_v)_*M\to \Upsilon_v M\to \Upsilon_v M(-1)\by{+1}
\end{equation}
hence the result follows from the uniqueness statement in Lemma \ref{l2}.
\end{proof}

\begin{cor} $L^\near_v(M,s)$ is a rational function in $N(v)^{-s}$, whose zeroes and poles are $N(v)$-Weil numbers.\qed
\end{cor}

This is remarkable because, in Lemma \ref{l2}, $g(s)$ is in general by no means a rational function of $p^{-s}$ when $f(s)$ is. (Take $f(s)=(1-p^{-s})^{-1}$.)

\begin{rk}\label{r9.1} Another argument, which was our initial argument, is to go via the $l$-adic realisation: one has
\[L_v^\tot(M,s) = L_v(i_v^*R(j_v)_*R^l(M),s).\]

If $V$ is an $l$-adic representation of $G_K$, we need to show that
\[L(i_v^*R(j_x)_* V,s)= f(N(v)^{-s})/f(N(x)^{-s-1})\]
for some $f\in \Q(t)$. 

We have
\[L(i_v^*R(j_v)_* V,s)= \frac{\det(1-\phi_vN(v)^{-s}\mid H^1(I_v,V))}{\det(1-\phi_vN(v)^{-s}\mid H^0(I_v,V))}.\]

Since $cd_l(I_v)=1$ this is an Euler-Poincar\'e characteristic, so we may assume $V$ \emph{semi-simple}. Then $I_v$ acts through a finite quotient by the $l$-adic monodromy theorem \cite[Appendix]{serre-tate}, thus
\[H^1(I_v,V) = V_{I_v}(-1) \simeq V^{I_v}(-1)\]
and
\[L(i_x^*R(j_v)_* V,s)= \frac{L^\Serre(V^{ss},s)}{L^\Serre(V^{ss},s+1)}\]
where $V^{ss}$ is the semi-simplification of $V$. We thus get another formula for $L^\near(M,s)$:
\begin{equation}\label{eq9.3}
L^\near(M,s)=L^\Serre(R^l(M)^{ss},s).
\end{equation}

Because of the finiteness of the action of the inertia, one might think of the right hand side of \eqref{eq9.3} as an \emph{Artin} $L$-function.
\end{rk}

\begin{qns} Composing with the functor $\Phi$ of \S \ref{s2.E}, we may associate a nearby $L$-function to any Chow motive, and this determines $L^\near$ by Bondarko's theorem.\\ 
1) By the $l$-adic realisation, this definition factors through homological equivalence. Does it even factor through numerical equivalence, as in positive characteristic? (The two $K_0$'s agree under the sign conjecture.)\\
2) Similarly, the operator $\Upsilon_v$ induces a homomorphism 
\[\Upsilon_v:K_0(\Chow(K))\to K_0(\Chow(\kappa(v)).\] 
To what extent can one describe it explicitly?
\end{qns}

\subsection{Examples}\label{s9.1} We finish with explicit computations. 

\subsubsection{Artin motives} In what comes before, we worked with motives with $\Q$-coefficients, but everything works just as well for motives with coefficients in a $\Q$-algebra, for example in a number field. This allows us to consider the nearby $L$-function $L^\near(\rho,s)$ attached to a complex Galois representation $\rho$. Then the action of inertia is semi-simple, thus, by \eqref{eq9.3} we have $L^\near(\rho,s)=L(\rho,s)$, the Artin $L$-function of $\rho$. So nothing new happens here.

\subsubsection{Elliptic curves}\label{s9.D.2}Let $E$ be an elliptic curve over $K$ with multiplicative reduction at $v$, $V= H^1_l(E)$. By hypothesis, the action of $I_v$ on $V$ is unipotent and nontrivial, hence $\dim V^{I_v}=1$ and $I_v$ acts trivially on $V/V^{I_v}$; thus $V^{ss}=V^{I_v}\oplus V/V^{I_v}$ and
\begin{align*}
L_v^\Serre(h^1(E),s) &= \det(1-N(v)^{-s}\phi_x\mid V^{I_v})^{-1}\\
L_v^\near(h^1(E),s) &= L_v^\Serre(H^1(E),s)\times \det(1-N(v)^{-s}\phi_v\mid V/V^{I_v})^{-1}.
\end{align*}

Extra poles thus are explicitly computable: if the multiplicative reduction is split, then $L_v^\Serre(h^1(E),s)=(1-N(v)^{-s})^{-1}$ and thus the other factor is $(1-N(v)^{1-s})^{-1}$, since the determinant is $H^2_l(E)\simeq \Q_l(-1)$. Similarly, if the multiplicative reduction is not split, the other factor is $(1+N(v)^{1-s})^{-1}$. 

Note that these factors have a functional equation between $s$ and $2-s$; so, such functional equations for $L^\Serre(h^1(E),s)$ and $L^\near(h^1(E),s)$ are equivalent. Similarly, the Beilinson conjectures for the first function easily imply Beilinson-like conjectures for the second, and conversely.

\section{The functional equation in characteristic $p$}

Let $K=\F_q(C)$ for $C$ a smooth, projective, geometrically connected curve. 
 As usual we abbreviate $\F_q=:k$, $\eta =\Spec K$ and write $j:\eta\to C$ for the canonical immersion. If $U\subseteq C$ is a nonempty open subset, we factor $j$ as a composition
\[\eta\by{j_U} U\by{j'_U}C.\]

Let $M\in \DM_\gm(K,\Q)$: we want to compare the $L$-functions
\[L^\near(M^*,1-s)\text{ and } L^\near(M,s).\]

We know that $M\simeq j^*\sM$ for some $\sM\in \D(C)$; from the functional equation for $\zeta(\sM,s)$ (Theorem \ref{t3.1}) we can get an approximate functional equation for $L^\near(M,s)$, but we would like a precise formula.

The following first approach was suggested by Joseph Ayoub. There exists $U$ and  $\sM_U\in \D^\proj(U)$ such that $M=j_U^*\sM_U$. Let $\sM=(j'_U)_* \sM_U$ with $j'_U:U\inj C$, so that $M=j^*\sM$. For $x\in C$, we have by \eqref{eq9.1} and Theorem \ref{t9.1}:
\[L^\near_x(M,s)=
\begin{cases}
\zeta(i_x^*\sM_U,s) &\text{for $x\in U$}\\
\zeta(\Upsilon_x(M),s) &\text{for all $x\in C$}
\end{cases}
\]
hence
\begin{multline*}
L^\near(M,s)= \prod_{x\in U}\zeta(i^*_x\sM_U,s)\times \prod_{x\notin U} \zeta(\Upsilon_x(M),s)\\ 
=\zeta(\sM_U,s)\times \prod_{x\notin U} \zeta(\Upsilon_x(M),s) \\
= \zeta(\sM,s)\times \prod_{x\notin U} \zeta(i_x^*(j'_U)_*\sM_U,s)^{-1} \times \prod_{x\notin U} \zeta(\Upsilon_x(M),s)\\
= \zeta(\sM,s)\times  \prod_{x\notin U} \zeta(\Upsilon_x(M),s-1)
\end{multline*}
by \eqref{eq9.2}.

On the other hand, $\sM^*:=\uHom(\sM,\Z)=(j'_U)_!\sM_U^*$, hence
\begin{multline*}
L^\near(M^*,1-s)= \prod_{x\in U}\zeta(i^*_x\sM_U^*,1-s)\times \prod_{x\notin U} \zeta(\Upsilon_x(M^*),1-s)\\ 
=\zeta(\sM^*,1-s)\times \prod_{x\notin U} \zeta(\Upsilon_x(M^*),1-s).
\end{multline*}

By \cite[Th. 11.16]{ayoubetale}, we have
\[\Upsilon_x(M^*)\simeq \Upsilon_x(M)^*   \]
where the duality on the right hand side is still relative to $\Z$, but in $\D(\kappa(x))$. Hence
\begin{multline*}
L^\near(M^*,1-s)=\zeta(\sM^*,1-s)\times \prod_{x\notin U} \zeta(\Upsilon_x(M)^*,1-s)\\
=(-q^{-s})^{\chi(f_!\sM)}\det(F_{f_!\sM})^{-1}\zeta(\sM,s)\\
\times \prod_{x\notin U} (-q_x^{1-s})^{\chi(\Upsilon_x(M))}\det(F_{\Upsilon_x(M)})^{-1}\zeta(\Upsilon_x(M),s-1)
\end{multline*}
where $q_x = |\kappa(x)|=q^{\deg(x)}$. We finally get

\begin{thm}\label{t9.2} One has
\[\frac{L^\near(M^*,1-s)}{L^\near(M,s)}=A (-q)^{-Bs}
\]
with
\begin{align*}
A&=(-q)^{\sum\limits_{x\notin U} \deg(x) \chi(\Upsilon_x(M))}\left(\det(F_{f_!\sM})\times \prod_{x\notin U} \det(F_{\Upsilon_x(M)})\right)^{-1}\\
B&=\chi(f_!\sM))+\sum_{x\notin U} \deg(x) \chi(\Upsilon_x(M)).\qed
\end{align*}
\end{thm}

\begin{qn}\label{q10.1} Comparing with \cite[Formulas (6) and (7)]{serre}, one would like to relate at least the constant $B$ to a \emph{conductor} of $M$. This can be defined \emph{via} the $l$-adic realisation; can one prove that the conductor thus obtained does not depend on $l$?
\end{qn}

\end{document}